\documentclass[a4paper, 11pt]{article}
\RequirePackage{fullpage}

\usepackage{amsthm,amsmath}
\usepackage{booktabs,array}
\usepackage{subcaption}
\usepackage[usenames,dvipsnames]{xcolor}
\usepackage[pdftex,breaklinks,colorlinks,
    citecolor={BlueViolet}, linkcolor={Blue},urlcolor=Maroon]{hyperref}
\usepackage{tikz}
\usetikzlibrary{shapes}
\usetikzlibrary{arrows.meta}
\usetikzlibrary{quotes}
\usetikzlibrary{calc}
\usetikzlibrary{positioning}
\usepackage{charter,eulervm}
\usepackage{enumerate}
\usepackage[final,expansion=alltext,protrusion=true]{microtype}
\usepackage[left]{lineno} 
\usepackage[colorinlistoftodos,prependcaption,textsize=tiny]{todonotes}

\theoremstyle{plain} 
\newtheorem{theorem}{Theorem}[section]
\newtheorem{lemma}[theorem]{Lemma}
\newtheorem{corollary}[theorem]{Corollary}
\newtheorem{proposition}[theorem]{Proposition}

\newcommand{\stpath}[2]{$#1$--$#2$ path}

\tikzset{
  prefix after node/.style={prefix after command=\pgfextra{#1}},
  /semifill/ang/.initial=0,
  /semifill/upper/.initial=none,
  /semifill/lower/.initial=none,
  semifill/.style={
    circle, draw,
    prefix after node={
      \pgfqkeys{/semifill}{#1}
      \path let \p1 = (\tikzlastnode.north), \p2 = (\tikzlastnode.center),
                \n1 = {\y1-\y2} in [radius=\n1]
            (\tikzlastnode.\pgfkeysvalueof{/semifill/ang}) 
            edge[
              draw=none,
              fill=\pgfkeysvalueof{/semifill/upper},
              to path={
                arc[start angle=\pgfkeysvalueof{/semifill/ang}, delta angle=180]
                -- cycle}] ()
            (\tikzlastnode.\pgfkeysvalueof{/semifill/ang}) 
            edge[
              draw=none,
              fill=\pgfkeysvalueof{/semifill/lower},
              to path={
                arc[start angle=\pgfkeysvalueof{/semifill/ang}, delta angle=-180]
                -- cycle}] ();}}}
\tikzset{
  corner/.style  = {circle, semifill={lower=black!50}, draw= blue, inner sep=1.5pt},
  b-vertex/.style = {fill=RubineRed, diamond, draw=blue, inner sep=1.5pt},
  a-vertex/.style = {draw=blue, fill=cyan, inner sep=2pt}
}
\tikzstyle{filled vertex}  = [{circle,draw=blue,fill=black!50,inner sep=1.pt}]  
\tikzstyle{empty vertex}  = [{circle, draw, fill = white, inner sep=1.5pt, minimum width=1.5pt}]

\newcommand{\lp}{\ensuremath{\operatorname{lp}}}
\newcommand{\rp}{\ensuremath{\operatorname{rp}}}
\newcommand{\lk}{\ensuremath{\operatorname{lk}}}
\newcommand{\rk}{\ensuremath{\operatorname{rk}}}

\title{Characterization of Circular-arc Graphs: I. Split Graphs}
 \author{
   Yixin Cao\thanks{Department of Computing, Hong Kong Polytechnic University, Hong Kong, China.  \texttt{yixin.cao@polyu.edu.hk}.
   This work was done while Y.C.~was visiting the Jagiellonian University.
   } 
   \and Jan Derbisz\thanks{Theoretical Computer Science Department,
 Institute of Computer Science,
 Jagiellonian University.
 Research of this author was partially funded by Polish National Science Center (NCN) grant 2021/41/N/ST6/03671.
 }
   \and Tomasz Krawczyk\thanks{Faculty of Mathematics and Information Science, Warsaw University of Technology, Poland. \texttt{tomasz.krawczyk@pw.edu.pl}.}
 }
\date{}

\begin{document}
\maketitle
\begin{abstract}
  The most elusive problem around the class of circular-arc graphs is identifying all minimal graphs that are not in this class. The main obstacle is the lack of a systematic way of enumerating these minimal graphs. McConnell [FOCS 2001] presented a transformation from circular-arc graphs to interval graphs with certain patterns of representations. We fully characterize these interval patterns for circular-arc graphs that are split graphs, thereby building a connection between minimal split graphs that are not circular-arc graphs and minimal non-interval graphs. This connection enables us to identify all minimal split graphs that are not circular-arc graphs.
  As a byproduct, we develop a polynomial-time certifying recognition algorithm for circular-arc graphs when the input is a split graph.  
  
\end{abstract}

\setcounter{page}{1} 

\section{Introduction}\label{sec:intro}

A graph is a \emph{circular-arc graph} if its vertices can be assigned to arcs on a circle such that two vertices are adjacent if and only if their corresponding arcs intersect.  Such a set of arcs is called a \emph{circular-arc model} for this graph (Figure~\ref{fig:normal-and-helly}).
If we replace the circle with the real line and arcs with intervals, we end with interval graphs.
All interval graphs are circular-arc graphs.
Both graph classes are by definition \emph{hereditary}, i.e., closed under taking induced subgraphs.

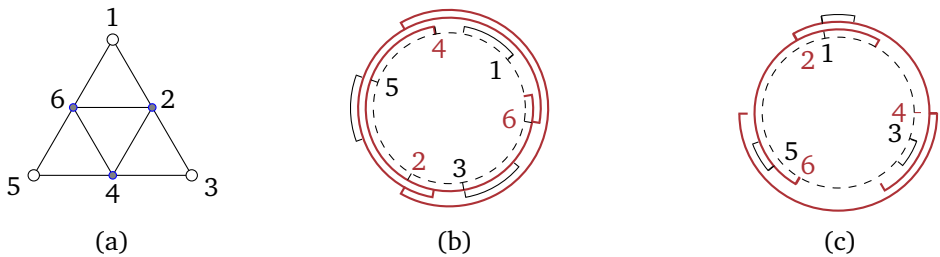
\begin{figure}[h!]
  \centering\small
  \begin{subfigure}[b]{.23\linewidth}
    \centering
    \begin{tikzpicture}[scale=.6]\small
      \foreach \i in {1, 2, 3} 
      \draw ({120*\i-90}:1) -- ({120*\i+30}:1) -- ({120*\i-30}:2) -- ({120*\i-90}:1);
      \foreach[count =\i from 0] \n in {1, 3, 5} {
        \node[empty vertex] at ({90 - 120*\i}:2) {};
        \node at ({90 - 120*\i}:2.5) {$\n$};
        \pgfmathsetmacro{\x}{int(\n+1)}%
        \node[filled vertex] at ({30 - 120*\i}:1) {};
        \node at ({30 - 120*\i}:1.4) {$\x$};
      }
    \end{tikzpicture}
    \caption{}
  \end{subfigure}  
  \;
  \begin{subfigure}[b]{.3\linewidth}
    \centering
    \begin{tikzpicture}[scale=.1]
      \begin{scope}[every path/.style={{|[left]}-{|[right]}}]        
        \draw[Sepia,thick] (120:13) arc (120.:-120:13); \draw (200:13) arc (200:160:13);
        \draw[Sepia,thick] (260:12) arc (260:-10:12); \draw (-40:12) arc (-40.:-80:12);
        \draw[Sepia,thick] (370:11) arc (370:100:11); \draw (80:11) arc (80:40:11);
      \end{scope}
      \draw[dashed,thin] (10,0) arc (0:360:10);
      \foreach[count=\j] \i/\c in {40/, -120/Sepia, -80/, 100/Sepia, 160/, -10/Sepia} {
        \draw[dashed] (\i:11) -- (\i:10);
        \node[\c] at (\i:8) {${\j}$};
      }
    \end{tikzpicture}
    \caption{}
  \end{subfigure}  
  \;
  \begin{subfigure}[b]{.3\linewidth}
    \centering
    \begin{tikzpicture}[scale=.1]
      \begin{scope}
        \foreach[count=\j from 0] \p/\q in {6/3, 2/5, 4/1} {
          \pgfmathsetmacro{\radius}{11+\j}
          \pgfmathsetmacro{\span}{180}
          \pgfmathsetmacro{\start}{240 - 120*\j}
          \draw[{|[left]}-{|[right]}, thick, Sepia]  (\start:\radius) arc (\start:{\start-\span}:\radius); 
          \draw[Sepia, dashed] ({\start}:\radius) -- ({\start}:10);
          \node[Sepia] at ({\start}:8) {${\p}$};

          \pgfmathsetmacro{\span}{20}
          \pgfmathsetmacro{\start}{340 - 120*\j}
          \draw[{|[left]}-{|[right]}]  (\start:\radius) arc (\start:{\start-\span}:\radius);
          \draw[dashed] ({\start}:\radius) -- ({\start}:10);
          \node at ({\start}:8) {${\q}$};
        }
      \end{scope}
      \draw[dashed,thin] (10,0) arc (0:360:10);
    \end{tikzpicture}
    \caption{}
  \end{subfigure}  
  \caption{\small A circular-arc graph and its two circular-arc models.  In (b), any two arcs for vertices~$\{2, 4, 6\}$ cover the circle; in (c), the three arcs for vertices~$\{2, 4, 6\}$ do not share any common point.}
  \label{fig:normal-and-helly}
\end{figure}

While both classes have been intensively studied, there is a huge gap between our understanding of them.
One fundamental combinatorial problem on a hereditary graph class is its characterization by \emph{forbidden induced subgraphs}, i.e., minimal graphs that are not in the class.
For example, the {forbidden induced subgraphs} of interval graphs are holes (induced cycles of length at least four) and those in Figure~\ref{fig:non-interval}~\cite{lekkerkerker-62-interval-graphs}.  The same problem on circular-arc graphs, however, has been open for sixty years~\cite{hadwiger-64-combinatorial-geometry, klee-69-cag}.

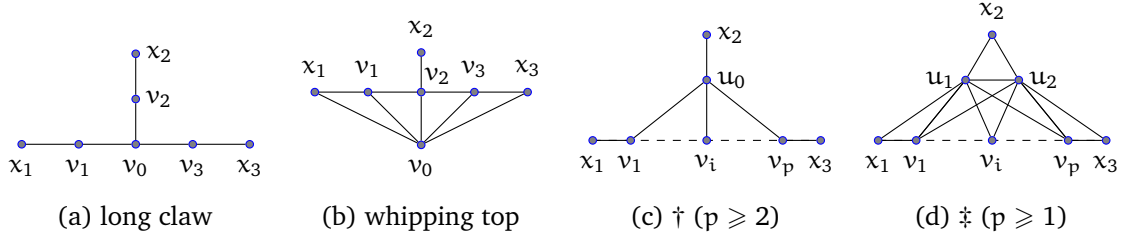
\begin{figure}[h]
  \tikzstyle{every node}=[filled vertex]
  \centering \small
  \begin{subfigure}[b]{0.22\linewidth}
    \centering
    \begin{tikzpicture}[every node/.style={filled vertex}, xscale=.75, yscale=.6]
      \node["$v_{0}$" below] (c) at (0, 0) {};
      \foreach[count=\i] \p in {below, right, below} {
        \node["$x_{\i}$" \p] (u\i) at ({90*(3-\i)}:2) {};
        \node["$v_{\i}$" \p] (v\i) at ({90*(3-\i)}:1) {};
        \draw (u\i) -- (v\i) -- (c);
      }
    \end{tikzpicture}
    \caption{long claw}\label{fig:long-claw}
  \end{subfigure}
  \,
  \begin{subfigure}[b]{0.22\linewidth}
    \centering
    \begin{tikzpicture}[every node/.style={filled vertex}, scale=.7]
      \node["$v_{0}$" below] (v7) at (0, -1) {};
      \draw (-2, 0) -- (2, 0);
      \foreach[count=\i from 2] \v/\p in {x_{1}/above, v_{1}/above, v_{2}/above right, v_{3}/above, x_{3}/above} {
        \node["$\v$" \p] (v\i) at ({\i - 4}, 0) {};
        \draw (v\i) -- (v7);
      }
      \node["$x_{2}$"] (v1) at (0, .75) {};
      \draw (v4) -- (v1);      
    \end{tikzpicture}
    \caption{whipping top}\label{fig:whipping-top}
  \end{subfigure}
  \,
  \begin{subfigure}[b]{0.22\linewidth}
    \centering
    \begin{tikzpicture}[every node/.style={filled vertex}, yscale=.8]
      \draw[dashed] (-1.5, 0) -- (1.5, 0);
      \draw (0, 1.75) node["$x_{2}$" right] {} -- (0, 1) node["$u_{0}$" right] (c) {} -- (0, 0) node["$v_{i}$" below] {};
      \foreach[count =\i] \x/\y in {1/1, p/3} {
        \node["$x_{\y}$" below] (u\i) at ({3 * \i - 4.5}, 0) {};
        \node["$v_{\x}$" below] (v\i) at ({2 * \i - 3}, 0) {};
        \draw (u\i) -- (v\i) -- (c);
      }
    \end{tikzpicture}
    \caption{\dag{} ($p\ge 2$)}\label{fig:dag}
  \end{subfigure}
  \,
  \begin{subfigure}[b]{0.22\linewidth}
    \centering
    \begin{tikzpicture}[every node/.style={filled vertex}, yscale=.8]
      \draw[dashed] (-1.5, 0) -- (1.5, 0);
      \node["$v_{i}$" below] (v3) at (0, 0) {};
      \foreach[count =\i, evaluate={\x=int(2*\i - 3);}] \a/\b in {1/1, p/3} {
        \node["$x_{\b}$" below] (u\i) at ({1.5 * \x}, 0) {};
        \node["$v_{\a}$" below] (v\i) at ({1. * \x}, 0) {};
        \node[label = {180*\i}:$u_{\i}$] (c\i) at ({.35 * \x}, 1) {};
        \draw (u\i) -- (v\i) -- (c\i) -- (u\i);
      }
      \foreach \i in {1, 2} {
        \foreach \j in {1, 2, 3}
        \draw (c\i) -- (v\j);
      }
      \draw (0, 1.75) node["$x_{2}$"] (x) {} -- (c1) -- (c2) -- (x);
    \end{tikzpicture}
    \caption{\ddag{} ($p\ge 1$)}\label{fig:ddag}
  \end{subfigure}
  \caption{Minimal chordal graphs that are not interval graphs.}
  \label{fig:non-interval}
\end{figure}

\begin{theorem}[\cite{lekkerkerker-62-interval-graphs}]
  \label{thm:lb}
  A graph~$G$ is an interval graph if and only if it does not contain any hole or any graph in Figure~\ref{fig:non-interval} as an induced subgraph.
\end{theorem}

It is already very complicated to characterize chordal circular-arc graphs by forbidden induced subgraphs.
\emph{Chordal graphs}, graphs in which all induced cycles are triangles, are another superclass of interval graphs.
Bonomo et al.~\cite{bonomo-09-partial-characterization-cag} characterized chordal circular-arc graphs that are claw-free.
Through generalizing Lekkerkerker and Boland's~\cite{lekkerkerker-62-interval-graphs} structural characterization of interval graphs, Francis et al.~\cite{francis-14-blocking-quadruple} defined a forbidden structure of circular-arc graphs.
This observation enables them to characterize chordal circular-arc graphs with independence number at most four.
As we will see, however, most minimal chordal graphs that are not circular-arc graphs contain claw, and their independence numbers can be arbitrarily large.

McConnell~\cite{mcconnell-03-recognition-cag} presented an algorithm that recognizes circular-arc graphs by transforming them into interval graphs.
Let~$G$ be a circular-arc graph and $\mathcal{A}$ a fixed arc model of~$G$.
If we \emph{flip} all arcs---replace arc~$[\lp, \rp]$ with arc~$[\rp, \lp]$---containing a certain point in~$\mathcal{A}$, we end with an interval model~$\mathcal{I}$.
In Figure~\ref{fig:normal-and-helly}b, for example, if we flip arcs~$2$, $4$, and~$6$, all containing the clockwise endpoint of the arc~$4$, we end with a $\dag$ graph of six vertices.
A crucial observation of McConnell~\cite{mcconnell-03-recognition-cag} is that the resulting interval graph is decided by the set~$K$ of vertices whose arcs are flipped and not by the original circular-arc models, as long as it is normalized (definition deferred to the next section).
It thus makes sense to denote it as~$G^K$.
He presented an algorithm to find a suitable set~$K$ and construct the graph directly from~$G$, without a circular-arc model.
As we will see, the construction is very simple when $G$ is chordal.
In particular, the closed neighborhood of every simplicial vertex can be used as the clique~$K$~\cite{hsu-95-independent-set-cag}.

However, $G^K$ being an interval graph does not imply that $G$ is a circular-arc graph.
The graph~$G$ is a circular-arc graph if and only if there is a clique~$K$ such that the graph~$G^K$ 
\begin{equation}
  \label{eq:1}\tag{$\sharp$}
  \parbox{\dimexpr\linewidth-4em}{%
    \strut
admits an interval model in which for every pair of vertices~$v\in K$ and~$u\in V(G)\setminus K$, the interval for~$v$ contains the interval for~$u$ if and only if they are not adjacent in~$G$.
    \strut
  }
\end{equation}


\begin{theorem}\label{thm:correlation}
  Let $G$ be a chordal graph.  The following are equivalent.
  \begin{enumerate}[i)]
  \item The graph~$G$ is a circular-arc graph.
  \item For every simplicial vertex~$s$, the graph~$G^{N[s]}$ satisfies \eqref{eq:1}.
  \item There exists a clique~$K$ such that the graph~$G^{K}$ satisfies \eqref{eq:1}.
  \end{enumerate}
\end{theorem}

We will use Theorem~\ref{thm:correlation} to derive a full characterization of minimal chordal graphs that are not circular-arc graphs.
In the present paper, we focus on a subclass of chordal graphs.
The rest will be left to a sequel paper. 
With few exceptions, most minimal chordal graphs that are not circular-arc graphs are closely related to the graphs in Figure~\ref{fig:split-non-helly} and the gadgets derived from them.

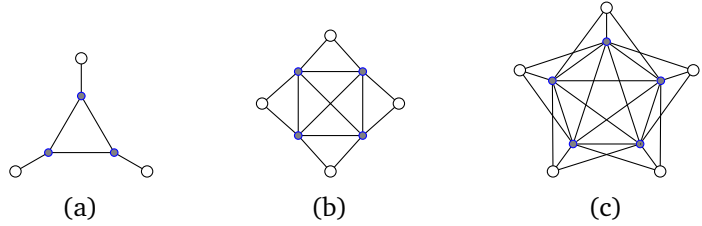
\begin{figure}[h]
  \centering \small
  \begin{subfigure}[b]{0.18\linewidth}
    \centering
    \begin{tikzpicture}[scale=.5]
      \foreach[count =\j] \i in {1, 2, 3} 
        \draw ({120*\i-30}:2) -- ({120*\i-30}:1) -- ({120*\i+90}:1);
        \foreach[count =\j] \i in {1, 2, 3} {
          \node[empty vertex] (u\i) at ({120*\i-30}:2) {};
          \node[filled vertex] (v\i) at ({120*\i-30}:1) {};
      }
    \end{tikzpicture}
    \caption{}
  \end{subfigure}
  \,
  \begin{subfigure}[b]{0.2\linewidth}
    \centering
    \begin{tikzpicture}[scale=.6]
      \def\n{4}
      \def\radius{1.5}
      \foreach \i in {1, ..., \n}
      \draw ({(\i - .5) * (360 / \n)}:1) -- ({(\i + .5) * (360 / \n)}:1) -- ({(\i) * (360 / \n)}:1.5) -- ({(\i - .5) * (360 / \n)}:1);
      \foreach \i in {1, ..., \n} {
        \node[empty vertex] (u\i) at ({(\i) * (360 / \n)}:1.5) {};
        \node[filled vertex] (v\i) at ({(\i - .5) * (360 / \n)}:1) {};
      }
      \draw (v1) -- (v3) (v2) -- (v4);
    \end{tikzpicture}
    \caption{}
  \end{subfigure}
  \quad
  \begin{subfigure}[b]{0.2\linewidth}
    \centering
    \begin{tikzpicture}[scale=.5]
      \def\n{5}
      \def\radius{1.5}      
      \foreach \i in {1,..., \n} {
        \pgfmathsetmacro{\angle}{90 - (\i) * (360 / \n)}
        \foreach \j in {-1, 0, 1}
        \draw (\angle:{\radius*1.6}) -- ({90 - (\i + \j) * (360 / \n)}:{\radius});
      }
      \foreach \i in {1,..., \n} {
        \pgfmathsetmacro{\angle}{90 - (\i) * (360 / \n)}
        \node[filled vertex] (v\i) at (\angle:\radius) {};
        \node[empty vertex] (u\i) at (\angle:{\radius*1.6}) {};
        \pgfmathsetmacro{\p}{\i - 1}        
        \foreach \j in {1, 2, ..., \p}
        \ifthenelse{\i>1}{\draw (v\i) -- (v\j)}{};
      }
    \end{tikzpicture}
    \caption{}
  \end{subfigure}
  \caption{The complements of~$k$-suns, for~$k = 3, 4, 5$.}
  \label{fig:split-non-helly}
\end{figure}

To describe our results, we need to mention an important subclass of circular-arc graphs.
A graph is a \emph{Helly circular-arc graph} if it admits a \emph{Helly circular-arc model}, where the arcs for every maximal clique have a shared point.
In Figure~\ref{fig:normal-and-helly}, e.g., the first model is Helly, and the second is not.
All interval models are Helly, and hence all interval graphs are Helly circular-arc graphs.
The characterization of Helly circular-arc graphs by forbidden induced subgraphs is also unknown, even restricted to chordal graphs.
We do know all minimal chordal circular-arc graphs that are not Helly circular-arc graphs.
For~$k \ge 2$, the \emph{$k$-sun}, denoted as~$S_{k}$, is the graph obtained from the cycle of length~${2 k}$ by adding all edges among the even-numbered vertices to make them a clique.  For example, $S_{3}$ and~$S_{4}$ are depicted in Figures~\ref{fig:normal-and-helly}a and \ref{fig:split-non-helly}b (note that the complement of $S_{4}$ is isomorphic to $S_{4}$).
Joeris et al.~\cite{joeris-11-hcag} proved that a chordal circular-arc graph is a Helly circular-arc graph if and only if it does not contain an induced copy of the complement of any $k$-sun, $k \ge 3$.
Interestingly, these forbidden induced subgraphs are all \emph{split graphs}, whose vertex sets can be partitioned into a clique and an independent set.\footnote{For~$k \ge 3$, the complement of~$S_{k}$ can be obtained by removing a Hamiltonian cycle from the complete split graph with~$k$ vertices on either side, in which all the edges between the clique and the independent set are present.}
The striking simplicity of split graphs may lead one to consider them as the ``simplest chordal graphs.''
However, Spinrad~\cite{spinrad-03-efficient-graph-representations} observed that ``\textit{split graphs ... often seem to be at the core of algorithms and proofs of difficulty for chordal graphs.}''
Indeed, the chordal forbidden induced subgraphs of the class of circular-arc graphs are natural generalizations of those within split graphs.

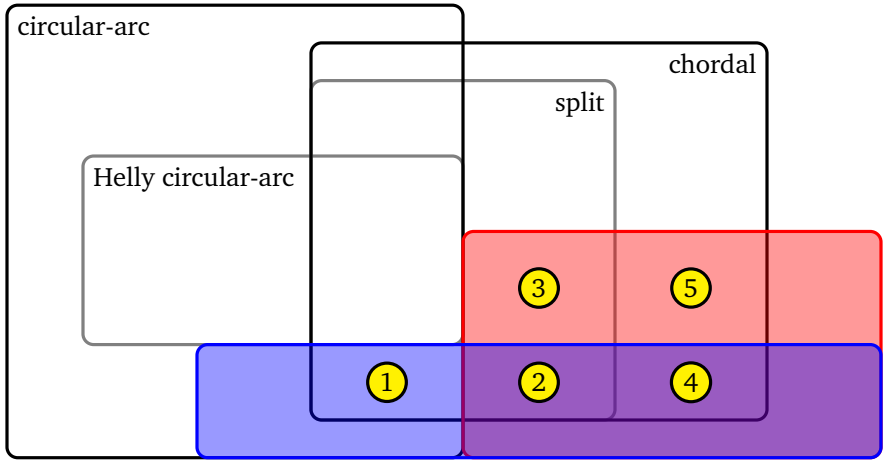
\begin{figure}[h]
  \centering \small

  \def\cag{(1, 0) rectangle ++(6, 6)}
  \def\hcag{(2, 1.5) rectangle ++(4.5, 4)}
  \def\chordal{(5, 0.5) rectangle ++(6, 4.5)}
  \def\split{(5, .5) rectangle ++(4, 4)}

  \def\firstregion{(3.5, 0) rectangle ++(9, 1.5)}  
  \def\secondregion{(7, 0) rectangle ++(5.5, 3)}  
  
  \begin{tikzpicture}[very thick,  text opacity=1]
    \draw[rounded corners, gray]\split;
    \draw[rounded corners, gray] \hcag;
    \draw[rounded corners]\chordal;
    \draw[rounded corners]\cag; 
    \begin{scope}[very thick, fill opacity=0.4, text opacity=1]
    \filldraw[red, rounded corners] \secondregion;
    \filldraw[blue, rounded corners] \firstregion;
    \end{scope}
    \path
    (11, 5) node[anchor=north east] {chordal}
    (9, 4.5) node[anchor=north east] {split}
    (1, 6) node[anchor=north west] {circular-arc}
    (2, 5.5) node[anchor=north west] {Helly circular-arc}
    (6, 1.) node[circle, draw, fill=yellow, inner sep=2pt] {1}
    (8, 1.) node[circle, draw, fill=yellow, inner sep=2pt] {2}
    (10, 1.) node[circle, draw, fill=yellow, inner sep=2pt] {4}
    (8, 2.25) node[circle, draw, fill=yellow, inner sep=2pt] {3}
    (10, 2.25) node[circle, draw, fill=yellow, inner sep=2pt] {5}
    ;
  \end{tikzpicture}
  \caption{The Venn diagram of the four graph classes.  The blue and red areas consists of \textit{minimal} forbidden induced subgraphs of the class of Helly circular-arc graphs and the class of circular-arc graphs, respectively.}
  \label{fig:venn-diagram}
\end{figure}

We use the Venn diagram in Figure~\ref{fig:venn-diagram} to illustrate the relationship of the four classes.  Regions 1, 2, and 4 together are the minimal chordal graphs that are not Helly circular-arc graphs, while regions 2--5 together are the minimal chordal graphs that are not circular-arc graphs.
The corresponding ones for split graphs are regions 1 and 2 and, respectively, regions 2 and 3.
Note that every graph in regions 3 and 5 contains a graph in region 1 as an induced subgraph.
As said, only region 1 has been fully understood~\cite{joeris-11-hcag}.

\begin{figure}[h]
  \centering \small
  \begin{subfigure}[b]{0.15\linewidth}
    \centering
    \begin{tikzpicture}[scale=.5]
      \foreach[count =\j] \i in {1, 2, 3} 
        \draw ({120*\i-30}:2) -- ({120*\i-30}:1) -- ({120*\i+90}:1);
        \foreach[count =\j] \i in {1, 2, 3} {
          \node[empty vertex] (u\i) at ({120*\i-30}:2) {};
          \node[filled vertex] (v\i) at ({120*\i-30}:1) {};
      }
    \end{tikzpicture}
    \caption{net}\label{fig:net}    
  \end{subfigure}
  \,
  \begin{subfigure}[b]{0.15\linewidth}
    \centering
    \begin{tikzpicture}[scale=.5]
      \foreach[count =\j] \i in {1, 2, 3} 
        \draw ({120*\i-90}:1) -- ({120*\i+30}:1) -- ({120*\i-30}:2) -- ({120*\i-90}:1);
        \foreach[count =\j] \i in {1, 2, 3} {
          \node[empty vertex] (u\i) at ({120*\i-30}:2) {};
        \node[filled vertex] (v\i) at ({120*\i-90}:1) {};
      }
    \end{tikzpicture}
    \caption{sun}\label{fig:sun}    
  \end{subfigure}
  \,
  \begin{subfigure}[b]{0.16\linewidth}
    \centering
    \begin{tikzpicture}[scale=.5]
      \node[empty vertex](a) at (0, 3) {};
      \draw (-1.8, 0) -- (1.8, 0);
      \foreach[count=\i] \x/\sub in {-1/1, 1/2} {
        \node[filled vertex]
        (b\i) at ({0.9*\x}, 1.5) {};
        \draw (a) -- (b\i);
        \foreach[count=\j]\t in {filled , empty } {
          \node[\t vertex](v\i\j) at ({0.9*\x*\j}, 0) {};
          \draw (b\i) -- (v\i\j);
        }
      }
      \draw (b1) -- (b2);
      \draw (b1) -- (v21) (b2) -- (v11);
    \end{tikzpicture}
    \caption{rising sun}\label{fig:rising-sun}    
  \end{subfigure}
  \caption{Minimal split graphs that are not interval graphs.}
  \label{fig:split-non-interval}
\end{figure}
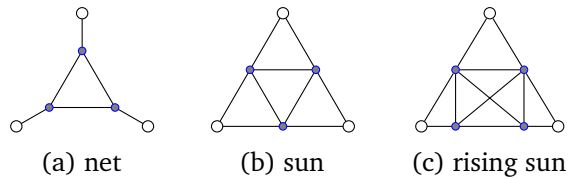

In the present paper, we focus on split graphs, i.e., regions 1--3.
Since each graph in them is not a Helly circular-arc graph, it is not an interval graph.
Thus, it contains a graph in Figure~\ref{fig:non-interval}.
Only three of them are split graphs, which are reproduced and named in Figure~\ref{fig:split-non-interval}.
The \emph{net} is a $\dag$ graph with six vertices, the \emph{sun} (i.e., $3$-sun) and the \emph{rising sun} are $\ddag$ graphs with six and seven vertices, respectively.
They are all circular-arc graphs.
We leave it to the reader to verify that the sun and the rising sun are Helly circular-arc graphs while the net is not.
Indeed, the net is the complement of the sun, hence in region 1.
Note that an $\overline{S_{4}}$ contains an induced rising sun, and an $\overline{S_{i}}, i \ge 5$, contains an induced sun.

For split graphs, Theorem~\ref{thm:correlation} can be simplified.
\begin{theorem}\label{thm:main-split}
  Let $G$ be a split graph with a split partition~$K\uplus S$.
  \begin{itemize}
  \item $G$ is a circular-arc graph if and only if
  there exists $s\in S$ such that
  $G^{N[s]}$ admits an interval model in which no interval for a vertex in~$N(s)$ contains an interval for a vertex in~$K\setminus N(s)$.
\item $G$ is a Helly circular-arc graph if and only if  $G^{K}$ is an interval graph.
  \end{itemize}
\end{theorem}

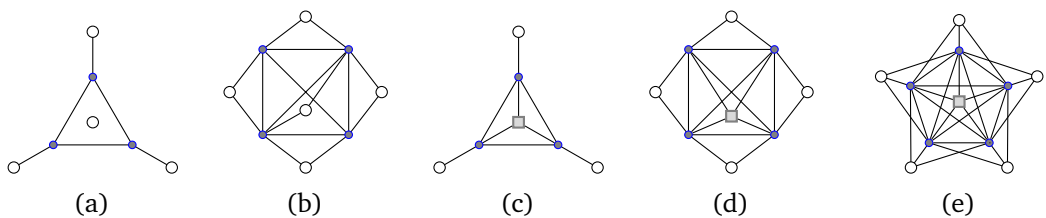
\begin{figure}[h]
  \centering \small
  \begin{subfigure}[b]{0.16\linewidth}
    \centering
    \begin{tikzpicture}[scale=.6]
      \node[empty vertex] (c) at (0, 0) {};
      \foreach[count =\j] \i in {1, 2, 3} 
      \draw ({120*\i-30}:2) -- ({120*\i-30}:1) -- ({120*\i+90}:1);
      \foreach[count =\j] \i in {1, 2, 3} {
        \node[empty vertex] (u\i) at ({120*\i-30}:2) {};
        \node[filled vertex] (v\i) at ({120*\i-30}:1) {};
      }
    \end{tikzpicture}
    \caption{}
    \label{fig:net-star}
  \end{subfigure}
  \,
  \begin{subfigure}[b]{0.16\linewidth}
    \centering
    \begin{tikzpicture}[scale=.8]
      \node[empty vertex] (c) at (0, -0.3) {};

      \def\n{4}
      \def\radius{1.25}
      
      \foreach \i in {1, ..., \n}
      \draw ({(\i - .5) * (360 / \n)}:1) -- ({(\i + .5) * (360 / \n)}:1) -- ({(\i) * (360 / \n)}:\radius) -- ({(\i - .5) * (360 / \n)}:1);
      \foreach \i in {1, ..., \n} {
        \node[empty vertex] (u\i) at ({(\i) * (360 / \n)}:\radius) {};
        \node[filled vertex] (v\i) at ({(\i - .5) * (360 / \n)}:1) {};
      }
      \foreach \i in {1, 3} \draw (c) -- (v\i);
      \draw (v1) -- (v3) (v2) -- (v4);
    \end{tikzpicture}
    \caption{}
    \label{fig:the-weird}
  \end{subfigure}
  \,
  \begin{subfigure}[b]{0.16\linewidth}
    \centering
    \begin{tikzpicture}[scale=.6]
      \node[corner] (c) at (0, 0) {};
      \foreach[count =\j] \i in {1, 2, 3} 
      \draw ({120*\i-30}:2) -- ({120*\i-30}:1) -- ({120*\i+90}:1);
      \foreach[count =\j] \i in {1, 2, 3} {
        \node[empty vertex] (u\i) at ({120*\i-30}:2) {};
        \node[filled vertex] (v\i) at ({120*\i-30}:1) {};
        \draw (c) -- (v\i);
      }
    \end{tikzpicture}
    \caption{}
  \end{subfigure}
  \,
  \begin{subfigure}[b]{0.16\linewidth}
    \centering
    \begin{tikzpicture}[scale=.8]
      \node[corner] (c) at (0, -0.4) {};

      \def\n{4}
      \def\radius{1.5}
      
      \foreach \i in {1, ..., \n}
      \draw ({(\i - .5) * (360 / \n)}:1) -- ({(\i + .5) * (360 / \n)}:1) -- ({(\i) * (360 / \n)}:1.25) -- ({(\i - .5) * (360 / \n)}:1);
      \foreach \i in {1, ..., \n} {
        \node[empty vertex] (u\i) at ({(\i) * (360 / \n)}:1.25) {};
        \node[filled vertex] (v\i) at ({(\i - .5) * (360 / \n)}:1) {};
        \draw (c) -- (v\i);
      }
      \draw (v1) -- (v3) (v2) -- (v4);
    \end{tikzpicture}
    \caption{}
  \end{subfigure}
  \quad
  \begin{subfigure}[b]{0.16\linewidth}
    \centering
    \begin{tikzpicture}[scale=.45]
      \node[corner] (c) at (0, 0) {};

      \def\n{5}
      \def\radius{1.5}      
      \foreach \i in {1,..., \n} {
        \pgfmathsetmacro{\angle}{90 - (\i) * (360 / \n)}
        \foreach \j in {-1, 0, 1}
        \draw (\angle:{\radius*1.6}) -- ({90 - (\i + \j) * (360 / \n)}:{\radius});
      }
      \foreach \i[evaluate={\p=int(\i - 1);}] in {1,..., \n} {
        \pgfmathsetmacro{\angle}{90 - (\i) * (360 / \n)}
        \node[filled vertex] (v\i) at (\angle:\radius) {};
        \node[empty vertex] (u\i) at (\angle:{\radius*1.6}) {};
        \foreach \j in {1, 2, ..., \p}
        \ifthenelse{\i>1}{\draw (v\i) -- (v\j)}{};
        \draw (c) -- (v\i);
      }
    \end{tikzpicture}
    \caption{}
  \end{subfigure}
  \caption{Minimal split graphs that are not circular-arc graphs.
    In (c--e), the square node can be from~$K$ or~$S$.}
  \label{fig:split-non-cag}
\end{figure}

\begin{figure}[h]
  \centering \small
  \begin{subfigure}[b]{0.2\linewidth}
    \centering
    \begin{tikzpicture}[scale=.6]
      \node[filled vertex] (c) at (0, 0) {};
      \foreach \i in {1, 2, 3} {
        \foreach[count=\j] \x in {0, 1} {
          \node[empty vertex](u\i\j) at ({120*\i - 30}:{1.5+\x/2}) {};
          \draw ({120*\i-90}:1) -- (u\i\j) -- ({120*\i+30}:1);
          \ifthenelse{\x=0}{\draw (u\i\j) -- (c)}{};
        }
          \draw ({120*\i-90}:1) -- ({120*\i+30}:1);
      }

      
      \foreach \i in {1, 2, 3} {
        \node[filled vertex](v\i) at ({120*\i-90}:1) {};
        \draw (v\i) -- (c);
      }
    \end{tikzpicture}
    \caption{}
    \label{fig:long-claw-derived}
  \end{subfigure}
  \,
  \begin{subfigure}[b]{0.2\linewidth}
    \centering
    \begin{tikzpicture}[scale=.6]
      \node[filled vertex] (c) at (0, 0) {};
      \foreach \i/\l/\p in {1/1/, 2/6/below, 3/2/below} {
          \node[empty vertex] (u\i) at ({120*\i-30}:2) {};
          \draw ({120*\i-90}:1) -- (u\i) -- ({120*\i+30}:1);
          \draw ({120*\i-90}:1) -- ({120*\i+30}:1);
      }
      \foreach[count =\j from 3] \i/\p in {1/, 2/below, 3/} {
        \node[filled vertex] (v\i) at ({120*\i+30}:1) {};
        \draw (v\i) -- (c);
      }
      \draw (u1) -- (c);
      \foreach \i/\x/\l in {1/-1/457, 3/1/347} {
        \draw (v\i) -- ({90 - 60*\x}:2) node[empty vertex] {};
      }
    \end{tikzpicture}
    \caption{}
    \label{fig:whipping-top-derived}
  \end{subfigure}

  \begin{subfigure}[b]{0.2\linewidth}
    \centering
    \begin{tikzpicture}[scale=.6]
      \draw[draw = gray] (0, .75) ellipse [x radius=1.8, y radius=0.85];
      \node[empty vertex] at (90:1.2) {};
      \foreach[count =\j] \i in {1, 2, 3} 
        \draw ({120*\i-90}:1) -- ({120*\i+30}:1) -- ({120*\i-30}:2) -- ({120*\i-90}:1);
        \foreach \i in {1, 2, 3} {
          \node[empty vertex] (u\i) at ({120*\i-30}:2) {};
          \node[filled vertex] (v\i) at ({120*\i-90}:1) {};
        }
        \foreach[count =\j] \i in {2, 1}
      \node at ({140*\j-120}:1.4) {$v_{\i}$};
    \end{tikzpicture}
    \caption{}
  \end{subfigure}
  \,
  \begin{subfigure}[b]{0.2\linewidth}
    \centering
    \begin{tikzpicture}[scale=.6]
      \draw[draw = gray] (0, .75) ellipse [x radius=1.8, y radius=0.9];
      \node[filled vertex] (c) at (0, 0) {};
      \foreach \i/\l/\p in {1/1/, 2/6/below, 3/2/below} {
          \node[empty vertex] (u\i) at ({120*\i-30}:2) {};
          \draw ({120*\i-90}:1) -- (u\i) -- ({120*\i+30}:1);
          \draw ({120*\i-90}:1) -- ({120*\i+30}:1);
      }
      \foreach[count =\j from 3] \i/\p in {1/, 2/below, 3/} {
        \node[filled vertex] (v\i) at ({120*\i+30}:1) {};
        \draw (v\i) -- (c);
        \draw (u\i) -- (c);
      }
      \foreach \i/\x/\l in {1/-1/457, 3/1/347} {
        \draw (v\i) -- ({90 - 40*\x}:1.5) node[empty vertex] {};
      }
      \foreach[count =\j] \i in {3, 1}
      \node at ({140*\j-120}:1.4) {$v_{\i}$};
    \end{tikzpicture}
    \caption{}
  \end{subfigure}
  \,
  \begin{subfigure}[b]{0.3\linewidth}
    \centering
    \begin{tikzpicture}[xscale=.8, yscale=.6]
      \draw[draw = gray] (0, .8) ellipse [x radius=1.7, y radius=0.9];
      \foreach[count =\j] \i in {1, 2, 3} 
        \draw ({120*\i-90}:1) -- ({120*\i+30}:1) -- ({120*\i-30}:2) -- ({120*\i-90}:1);
        \foreach \i in {1, 2, 3} {
          \node[empty vertex] (u\i) at ({120*\i-30}:2) {};
          \node[filled vertex] (v\i) at ({120*\i-90}:1) {};
      }

    \foreach[count=\i] \x in {1, -1} {
      \node[filled vertex](x\i) at ({\x/2.5}, .05) {};
      \foreach \k in {1, 2, 3}  \draw (v\k) -- (x\i) -- (u\k);
    }
    \draw (x1) -- (x2);
    \foreach \i/\x in {1/1, 2/-1} {
      \draw (v\i) -- ({90 - 40*\x}:1.5) node[empty vertex] {} -- (x\i);
    }
    \draw (v1) -- (90:1.2) node[empty vertex] {} -- (v2);

      \foreach[count =\j] \i in {4, 1}
      \node at ({140*\j-120}:1.3) {$v_{\i}$};
    \end{tikzpicture}
    \caption{}
  \end{subfigure}

  \begin{subfigure}[b]{0.2\linewidth}
    \centering
    \begin{tikzpicture}[scale=.6]
      \draw[draw = gray] (0, 1.6) ellipse [x radius=1.8, y radius=0.75];
      \node[empty vertex] at (90:2.1) {};

      \node[empty vertex](a) at (0, 3) {};
      \draw (-1.8, 0) -- (1.8, 0);
      \foreach[count=\i] \x/\sub in {-1/1, 1/2} {
        \node[filled vertex, label={[xshift={\x*0.3cm}, yshift=-0.3cm]$v_{\sub}$}]
        (b\i) at ({0.9*\x}, 1.5) {};
        \draw (a) -- (b\i);
        \foreach[count=\j]\t in {filled , empty } {
          \node[\t vertex](v\i\j) at ({0.9*\x*\j}, 0) {};
          \draw (b\i) -- (v\i\j);
        }
      }
      \draw (b1) -- (b2);
      \draw (b1) -- (v21) (b2) -- (v11);
    \end{tikzpicture}
    \caption{}
  \end{subfigure}
  \,
  \begin{subfigure}[b]{0.2\linewidth}
    \centering
    \begin{tikzpicture}[scale=.6]
      \draw[draw = gray] (0, 1.7) ellipse [x radius=1.8, y radius=0.95];

      \node[empty vertex] (a3) at (0, 3) {};
      \draw (-1.8, 0) -- (1.8, 0);
      \draw (a3) -- (0, 1) node[filled vertex] (c) {};
      \foreach[count=\i] \x/\sub in {-1/1, 1/3} {
        \node[filled vertex, label={[xshift={\x*0.3cm}, yshift=-0.3cm]$v_{\sub}$}]
        (b\i) at ({0.9*\x}, 1.5) {};
        \draw (b\i) --  ({0.9*\x}, 2.2) node[empty vertex] {};
        \draw (a3) -- (b\i) -- (c);
        \foreach[count=\j]\t in {filled , empty } {
          \node[\t vertex](v\i\j) at ({0.9*\x*\j}, 0) {};
          \draw (b\i) -- (v\i\j) -- (c);
        }
      }
      \draw (b1) -- (b2);
      \draw (b1) -- (v21) (b2) -- (v11);
    \end{tikzpicture}
    \caption{}
  \end{subfigure}
  \,
  \begin{subfigure}[b]{0.3\linewidth}
    \centering
    \begin{tikzpicture}[xscale=.8, yscale=.6]
      \draw[draw = gray] (0, 1.75) ellipse [x radius=1.7, y radius=0.85];

      \node[empty vertex](a3) at (0, 3) {};
      \draw (-1.8, 0) -- (1.8, 0);
      \foreach[count=\i] \x/\sub in {-1/1, 1/4} {
        \node[filled vertex, label={[xshift={\x*0.3cm}, yshift=-0.3cm]$v_{\sub}$}]
        (b\i) at ({0.9*\x}, 1.5) {};
        \draw (a3) -- (b\i);
        \foreach[count=\j]\l/\t in {c/filled , a/empty } {
          \node[\t vertex](\l\i) at ({0.9*\x*\j}, 0) {};
          \draw (b\i) -- (\l\i);
        }
      }
      \draw (b1) -- (b2);
      \draw (b1) -- (c2) (b2) -- (c1);

    \foreach[count=\i] \x in {-1, 1} {
      \node[filled vertex](x\i) at ({\x/3}, 1.2) {};
      \foreach \k in {1, 2, 3}  \draw (a\k) -- (x\i);
      \foreach \k in {1, 2}  \draw (b\k) -- (x\i) -- (c\k);
    }
    \draw (x1) -- (x2);
    \foreach \i/\x in {1/-1, 2/1} {
      \draw (b\i) -- ({0.9*\x}, 2.2) node[empty vertex] {} -- (x\i);
    }
    \draw (b1) -- (90:2.2) node[empty vertex] {} -- (b2);
    \end{tikzpicture}
    \caption{}
  \end{subfigure}
\caption{Minimal split graphs that are not circular-arc graphs (region 2).  The graph~$G^{K}$, where $K$ comprises the solid nodes, are (a) long claw, (b) whipping top, (c--e) $\dag$ graphs on six, seven, and eight vertices, and (f--h) $\ddag$ graphs on seven, eight, and nine vertices.  For (c--h), inside the ellipse is
  a subgraph of~$\overline{S_{k}}, k = 2, 3, 4$, obtained by removing one simplicial vertex.
}
  \label{fig:split-non-helly-2}
\end{figure}
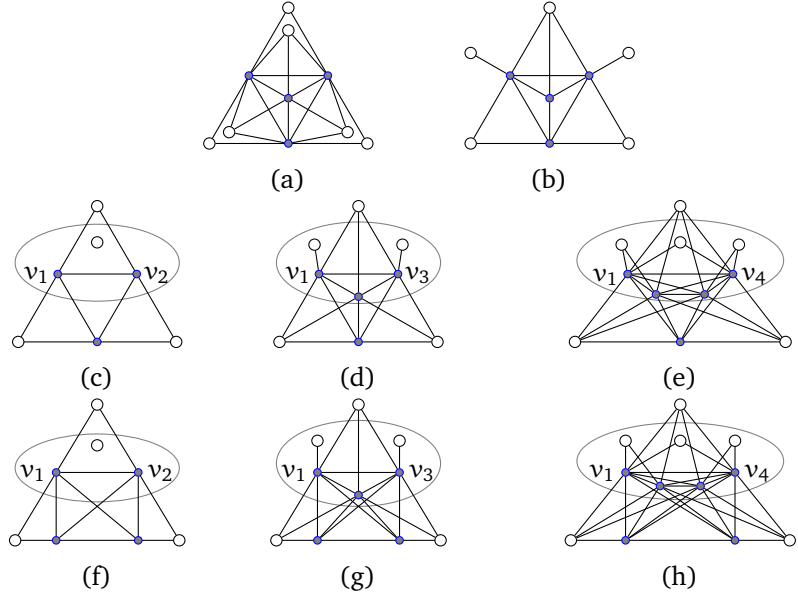

Theorem~\ref{thm:main-split} enables us to fully characterize regions 2 and 3.
For a split graph~$G$ with a unique split partition~$K\uplus S$, we use $G^+$ to denote the graph obtained from~$G$ by adding a vertex adjacent to all the vertices in~$K$.
Note that $\overline{S_{k}^{+}} = \left(\overline{S_{k}}\right)^{+}$, and the first three of them are shown as Figure~\ref{fig:split-non-cag}(c--e).
We also introduce two families of graphs~$S^{1}_{k}$ and~$S^{2}_{k}$ for each~$k\ge 2$, which will be defined in Section~\ref{sec:hcag}.  The first three of each family are shown in Figure~\ref{fig:split-non-helly-2}.

\begin{theorem}\label{thm:main}
  Region 2 comprises the graph in Figure~\ref{fig:long-claw-derived}, the graph in Figure~\ref{fig:whipping-top-derived},
 and~$S^{1}_{k}$, $S^{2}_{k}, k \ge 2$.
 Region 3 comprises the graph in Figure~\ref{fig:net-star}, the graph in Figure~\ref{fig:the-weird}, and~$\overline{S_{k}^{+}}, k \ge 3$.
\end{theorem}

\begin{corollary}
  A split graph is a Helly circular-arc graph if and only if it does not contain an induced copy of the graphs in Figure~\ref{fig:long-claw-derived}, \ref{fig:whipping-top-derived}, or~$S^{1}_{k}$, $S^{2}_{k}, k \ge 2$.

A split graph is a circular-arc graph if and only if it does not contain an induced copy of the graphs in
Figures~\ref{fig:net-star}, \ref{fig:the-weird}, \ref{fig:long-claw-derived}, \ref{fig:whipping-top-derived}, or~$S^{1}_{k}$, $S^{2}_{k}, \overline{S_{k+1}^{+}}, k \ge 2$.
\end{corollary}

\section{Preliminaries}

All graphs discussed in this paper are finite and simple.  The vertex set and edge set of a graph~$G$ are denoted by, respectively, $V(G)$ and~$E(G)$.
For a subset~$U\subseteq V(G)$, we denote by $G[U]$ the subgraph of~$G$ induced by~$U$, and by~$G - U$ the subgraph~$G[V(G)\setminus U]$, which is shortened to~$G - v$ when $U = \{v\}$.
The \emph{neighborhood} of a vertex~$v$, denoted by~$N_{G}(v)$, comprises vertices adjacent to~$v$, i.e., $N_{G}(v) = \{ u \mid uv \in E(G) \}$, and the \emph{closed neighborhood} of~$v$ is $N_{G}[v] = N_{G}(v) \cup \{ v \}$.
The \emph{closed neighborhood} and the \emph{neighborhood} of a set~$X\subseteq V(G)$ of vertices are defined as~$N_{G}[X] = \bigcup_{v \in X} N_{G}[v]$ and~$N_{G}(X) =  N_{G}[X] \setminus X$, respectively.
We may drop the subscript if the graph is clear from the context.

A \emph{clique} is a set of pairwise adjacent vertices, and an \emph{independent set} is a set of vertices that are pairwise nonadjacent. 
We say that a vertex~$v$ is \emph{simplicial} if $N[v]$ is a clique; such a clique is necessarily maximal. 
For~$\ell \ge 3$, we use $C_\ell$ to denote an induced cycle on~$\ell$ vertices; it is called an \emph{$\ell$-hole} if $\ell \ge 4$.
For a graph~$G$, we use $G^{\star}$ to denote the graph obtained from~$G$ by adding an isolated vertex. 
The \emph{complement graph}~$\overline{G}$ of a graph~$G$ is defined on the same vertex set~$V(G)$, where a pair of distinct vertices~$u$ and~$v$ is adjacent in~$\overline{G}$ if and only if $u v \not\in E(G)$.

For a chordal graph~$G$ and a (not necessarily maximal) clique~$K$ of~$G$, we define an \emph{auxiliary graph}~$G^{K}$ with the vertex set~$V(G)$.
The edges among vertices in~$V(G)\setminus K$ are the same as in~$G$.
A pair of vertices~$u, v\in K$ are adjacent in~$G^{K}$ if there exists a vertex adjacent to neither of them, i.e., $N_{G}(u)\cup N_{G}(v)\ne V(G)$.
A pair of vertices~$u\in K$ and~$v\in V(G)\setminus K$ are adjacent in $G^{K}$ if $N_{G}[v]\not\subseteq N_{G}[u]$.
Two quick remarks on the conditions are in order.
First, note that $N(u)\cup N(v) = V(G)$ if and only if $N[u]\cup N[v] = V(G)$ and~$uv\in E(G)$.
Second, $N_{G}[v]\not\subseteq N_{G}[u]$ if and only if either they are not adjacent, or there exists a vertex adjacent to~$v$ but not $u$ in~$G$.

A \emph{circular-arc graph} is the intersection graph of a set of arcs on a circle.  The set of arcs, together with the circle, constitutes a \emph{circular-arc model} of this graph.
A circular-arc model is \emph{normalized} if the following hold for every pair of adjacent vertices~$v_{1}$ and~$v_{2}$.
\begin{itemize}
\item The arc~$A(v_{1})$ contains the arc~$A(v_{2})$ whenever $N[v_{2}]\subseteq N[v_{1}]$.
\item The arcs~$A(v_{1})$ and~$A(v_{2})$ cover the circle whenever $N(v_{1})\cup N(v_{2}) = V(G)$ and~$v_{1} v_{2}$ is not an edge of a $C_{4}$.
\end{itemize}
It is known that a circular-arc graph admits a normalized circular-arc model if and only if it does not have any universal vertex~\cite{spinrad-88-case-1, hsu-95-independent-set-cag}.
Since a chordal graph has no holes, for the second condition, it suffices to check whether $A(v_{1})$ and~$A(v_{2})$ cover the circle whenever $N(v_{1})\cup N(v_{2}) = V(G)$.

\begin{lemma}\label{lem:construction}
  Let $G$ be a chordal circular-arc graph, and~$K$ a clique.
  If there is a normalized circular-arc model of~$G$ in which a point is covered by and only by arcs for vertices in~$K$, then the graph~$G^{K}$ satisfies \eqref{eq:1}.
\end{lemma}
\begin{proof}
  Let $\mathcal{A}$ be a normalized circular-arc model of~$G$ in which a point is covered by and only by arcs for vertices in~$K$, and two vertices~$x, y$ share an endpoint if and only if $N[x] = N[y]$ (then the definition of normalized models forces $x$ and~$y$ to have the same arc).
  Denote by~$\lp(x)$ and~$\rp(x)$, respectively, the counterclockwise and clockwise endpoints of the arc~$A(x)$.
  We produce an interval model by setting
  \[
    I(x) =
    \begin{cases}
      [\rp(x), \lp(x)] & \text{if } x\in K,
      \\
      [\lp(x), \rp(x)] & \text{if } x\in V(G)\setminus K.
    \end{cases}
  \]
  We show that a pair of vertices~$x$ and~$y$ are adjacent in~$G^{K}$ if and only if $I(x)\cap I(y)\ne \emptyset$, and when $x\in K$ and~$y\in V(G)\setminus K$,
  \[
    I(y)\subseteq I(x) \Leftrightarrow x y\not\in E(G).
  \]
  
  
  First, if both $x$ and~$y$ are in~$V(G)\setminus K$, then $I(x) \cap I(y) = A(x)\cap A(y)$, which is empty if and only if $x y\not\in E(G)$ and~$x y\not\in E(G^{K})$.

  Second, suppose that $x, y\in K$.  By construction, $x y\in E(G^{K})$ if and only if $N[x]\cup N[y] \ne V(G)$.
  If $N[x]\cup N[y] = V(G)$, then $A(x)$ and~$A(y)$ cover the circle, and hence $I(x)$ and~$I(y)$ are disjoint. 
  Otherwise, there exists a vertex~$z\in V(G)\setminus K$ that is adjacent to neither of them in~$G$, and~$I(z) = A(z) \subseteq I(x)\cap I(y)$.
  
  Finally, suppose that $x\in K$ and~$y\in V(G)\setminus K$.
  If $x y\not\in E(G)$, then $I(y)\subset I(x)$ and~$x y\in E(G^{K})$.
  In the rest, $x y\in E(G)$.
  Since $A(x)\cap A(y)\ne \emptyset$ and their endpoints are distinct, $I(y)\not\subset I(x)$.
  Note that
   $x y\in E(G^{K})$ if and only if there exists a vertex~$z\in N_{G}(y)\setminus N_{G}(x)$.
  If such a $z$ exists, then $I(y) \cap I(z)\subseteq I(z) \subset I(x)$, and hence $I(x) \cap I(y)\ne \emptyset$.
  Otherwise, $N_{G}[y]\subset N_{G}[x]$.
  Since the model~$\mathcal{A}$ is normalized, $A(y)\subseteq A(x)$, which means that $I(x)\cap I(y)=\emptyset$.  
\end{proof}


A trivial choice for the clique~$K$ required by Lemma~\ref{lem:construction} is the closed neighborhood of a simplicial vertex. 
The following is not restricted to chordal graphs, but a chordal graph guarantees the existence of simplicial vertices.
In general, it is rather challenging to locate such a clique, which is step~2 of McConnell's algorithm~\cite[Section~8]{mcconnell-03-recognition-cag}.

\begin{lemma}\label{lem:simplicial-vertices}
  Let $G$ be a circular-arc graph, and let $s$ be a simplicial vertex of~$G$.
  In any normalized model of~$G$, there is a point covered by and only by arcs for~$N[s]$.
\end{lemma}
\begin{proof}
  Let $x$ be a neighbor of~$s$.
  By definition, $A(s)\subseteq A(x)$ because $N[s] \subseteq N[x]$.  Thus, any point in~$A(s)$ is covered precisely by the arcs for~$N[s]$.
\end{proof}


\begin{proposition}\label{lem:trivial}
  Let $G$ be a chordal graph, $K$ a clique, and~$x$ a simplicial vertex of~$G$. 
 If $x\not\in K$, then it is also simplicial in~$G^{K}$.
\end{proposition}
\begin{proof}

  By construction, $N_{G^{K}}(x)\setminus K = N_{G}(x)\setminus K$, and it is a clique of~$G^{K}$.
  On the other hand, $N_{G^{K}}(x)\cap K = K\setminus N_{G}(x)$.
  By construction, any two vertices in~$N_{G^{K}}(x)\setminus K$ and~$N_{G^{K}}(x)\cap K$ are adjacent and any two vertices in~$N_{G^{K}}(x)\cap K$ are adjacent.
  Thus, $x$ is simplicial in~$G^{K}$.
\end{proof}

We are now ready to prove the main theorem of this section.
\begin{proof}[Proof of Theorem~\ref{thm:correlation}]
  The implication from (i) to (ii) follows from Lemmas~\ref{lem:construction} and~\ref{lem:simplicial-vertices}.
  Since $G$ contains at least one simplicial vertex~$s$, and~$N_{G}[s]$ is a clique, (ii) implies (iii).  
  In the rest, we show (iii) implies (i).  Suppose that $\mathcal{I} = \{ [\lp(x), \rp(x) ] \mid x\in V(G)\}$ is an interval model of~$G^{K}$ specified in \eqref{eq:1}.
  We may assume that all the endpoints in~$\mathcal{I}$ are positive, and hence no interval contains the point~$0$.
  We claim that the following arcs on a circle of length~$\ell + 1$, where $\ell$ denotes the maximum of the $2n$ endpoints in~$\mathcal{I}$, gives a circular-arc model of~$G$:
  \[
    A(x) =
    \begin{cases}
      [\rp(x), \lp(x)] & \text{if } x\in K,
      \\
      [\lp(x), \rp(x)] & \text{if } x\in V(G)\setminus K.
    \end{cases}
  \]
  All the arcs for vertices in~$K$ intersect because they cover the point~$0$.
  On the other hand, note that $G^{K} - K = G - K$, while $I(x) = A(x)$ for all $x\in V(G)\setminus K$.
  For a pair of vertices~$x\in K$ and~$y\in V(G)\setminus K$,
  by assumption, $x y\not\in E(G)$ if and only if $I(y)\subset I(x)$, which is equivalent to~$A(x)\cap A(y) \ne \emptyset$.
\end{proof}

A graph is a \emph{split graph} if its vertex set can be partitioned into a clique~$K$ and an independent set~$S$.  We use $K \uplus S$ to denote this partition, called a \emph{split partition}.
%
A $k$-sun has a unique split partition, and so does its complement, but the complete graph on~$n$ vertices has $n + 1$ split partitions.
We say that a split graph is \emph{ambiguous} if there are at least two different split partitions.

\begin{lemma}\label{lem:special-max-clique}
  A split graph~$G$ is ambiguous if and only if every maximal clique of~$G$ contains a simplicial vertex.
\end{lemma}
\begin{proof}
  Suppose that there exists a maximal clique~$K$ containing no simplicial vertices.
  Let $K'\uplus S'$ be a split partition of~$G$.  Since each vertex in~$S'$ is simplicial, $K\subseteq K'$.
  The maximality of~$K$ implies $K = K'$.  Thus, $G$ is not ambiguous.
  Now suppose that $G$ is unambiguous, and let $K\uplus S$ be its unique split partition.
  If $K$ is a proper subset of another clique~$K'$, then $K'\uplus (V(G)\setminus K')$ is a different split partition.
  If any vertex~$v$ in~$K$ is simplicial, then $v$ has at most one neighbor in~$S$.
  Moreover, if $x\in N(v)\cap S$, then $K\subseteq N(x)$.
  In either case, we can find a different split partition, a contradiction.
  Thus, $K$ is a maximal clique containing no simplicial vertices.
\end{proof}

In an ambiguous split graph $G$, every maximal clique is $N[s]$ for some simplicial vertex~$s$.
Thus, if $G$ is a circular-arc graph, it has to be a Helly circular-arc graph by Lemma~\ref{lem:simplicial-vertices}.
In other words, if a split graph~$G$ is a circular-arc graph but not a Helly circular-arc graph, then it must be unambiguous.

\begin{theorem}\label{lem:split-cag-and-hcag}
  An ambiguous split graph is a Helly circular-arc graph if and only if it is a circular-arc graph.
\end{theorem}
\begin{proof}
The necessity is trivial, while the sufficiency follows from Lemmas~\ref{lem:special-max-clique} and~\ref{lem:simplicial-vertices}.
\end{proof}

\section{Split graphs that are not Helly circular-arc graphs}
\label{sec:hcag}

Let $G$ be a split graph, and~$K\uplus S$ a split partition of~$G$.
The graph~$G^{K}$ has a very simple structure: $S$ remains an independent set; a pair of vertices~$v\in K$ and~$x\in S$ are adjacent in~$G^{K}$ if and only if they are not adjacent in~$G$;
and a pair of vertices~$u, v\in K$ are adjacent in~$G^{K}$ if and only if $N_{G}(u)\cup N_{G}(v)\ne V(G)$, i.e., there exists a vertex adjacent to neither of them.
If $G^{K}$ is an interval graph, then it satisfies \eqref{eq:1} by Proposition~\ref{lem:trivial}.
For a split graph that is a Helly circular-arc graph, we have thus a simpler statement than Theorem~\ref{thm:correlation}.

\begin{theorem}\label{thm:helly-main}
  A split graph~$G$ is a Helly circular-arc graph if and only if $G^{K}$ is an interval graph.
\end{theorem}
\begin{proof}
  The necessity follows from Lemma~\ref{lem:construction}.
  All normalized circular-arc models of~$G$ are Helly~\cite{joeris-11-hcag}.
  Thus, if $K$ is maximal, then in any normalized circular-arc model of~$G$, there is a point corresponding to~$K$.
  Otherwise, there is a vertex~$x\in S$ such that $K\subseteq N(x)$.
  By Lemma~\ref{lem:simplicial-vertices}, there is a point covered by and only by arcs for~$K\cup \{x\}$.
  We can take a point immediately after the right endpoint of the arc for~$x$.
  For sufficiency, we may assume without loss of generality that there is no universal vertex (note that a universal vertex of~$G$ is isolated in~$G^{K}$).
  We create a new graph~$G'$ by adding a vertex~$s$ to~$G$ and making it adjacent to~$K$.
  The vertex~$s$ is universal in~$(G')^{K\cup \{s\}}$, and~$(G')^{K\cup \{s\}} - s$ is isomorphic to~$G^K$.  Thus, $(G')^{K\cup \{s\}}$ is an interval graph.
  In~$(G')^{K\cup \{s\}}$, all vertices in~$S$ are simplicial (Proposition~\ref{lem:trivial}), and two vertices~$v\in K$ and~$u\in V(G)\setminus K$ are adjacent if and only if they are not adjacent in~$G$.  Thus, any interval model of~$(G')^{K\cup \{s\}}$ vacuously satisfies \eqref{eq:1}, and~$G'$ is a circular-arc graph by Theorem~\ref{thm:correlation}.
  Since $G'$ is ambiguous, it is a Helly circular-arc graph by Theorem~\ref{lem:split-cag-and-hcag}.  Then $G$ is a Helly circular-arc graph because it is an induced subgraph of~$G'$.
\end{proof}

Throughout this section, we may assume without loss of generality that there are no universal vertices in~$G$.
Note that for any pair of vertices~$u, v\in K$ that are adjacent in~$G^{K}$, the vertices in~$V(G)\setminus (N_{G}(u)\cup N_{G}(v))$ can be viewed as ``witnesses'' of the edge~$u v$; they must be from~$S$.
We can generalize this observation to any clique~$X$ of~$G^{K}[K]$.
A vertex~$w\in S$ is a \emph{witness} of~$X$ if $w$ is adjacent to all the vertices in~$X$ in~$G^{K}$, i.e.,
$X\subseteq V(G)\setminus N_{G}(w)$.
The clique~$X$ is then \emph{witnessed}.

\begin{proposition}\label{lem:h-maxclique}
  If a clique of~$G^{K}[K]$ is not witnessed, 
  then $G^{K}$ contains an induced sun of which all the degree-two vertices are from~$S$.
\end{proposition}
\begin{proof}
  Suppose that $K'$ is a smallest unwitnessed clique of~$G^{K}[K]$, 
  Note that $|K'| > 2$: by assumption, $G$ does not contain a universal vertex; a clique of order two is an edge and hence has a witness by construction.
  We take three vertices~$v_{1}$, $v_{2}$, and~$v_{3}$ from~$K'$.
  By the selection of~$K'$, for each $i = 1, 2, 3$, the clique~$K'\setminus \{v_{i}\}$ has a witness~$x_{i}\in S$.
  Since $x_{i}$ is not a witness of~$K'$, it cannot be adjacent to~$v_{i}$ in~$G^{K}$.
  Then $\{v_{1}, v_{2}, v_{3}, x_{1}, x_{2}, x_{3}\}$ induces a sun in~$G^{K}$.
\end{proof}

We say an induced subgraph of~$G^{K}$ is \emph{witnessed} if every maximal clique of this subgraph disjoint from~$S$ is witnessed.
Note that all holes of~$G^{K}$ are trivially witnessed because they do not have any clique of order three.

\begin{proposition}\label{lem:simplicial-iff-i-helly}
  If $G^{K}$ contains a witnessed minimal non-interval induced subgraph~$F$, then it contains an induced and witnessed subgraph isomorphic to~$F$ 
  all simplicial vertices of which are from~$S$.
\end{proposition}
\begin{proof}
  We are done if all simplicial vertices of~$F$ are from~$S$.
  Suppose that a simplicial vertex~$v$ of~$F$ is from~$K$.  
 A check of Figure~\ref{fig:non-interval} convinces us
 that no vertex in~$N_{F}(v)$ can be simplicial in~$G^{K}$.  Thus, $N_{F}(v)\subseteq K$ by Proposition~\ref{lem:trivial}.
 If $v\in K$, there is a witness~$w$ of~$N_{F}[v]$ by assumption.
  By Proposition~\ref{lem:trivial}, the vertex~$w$ has no other neighbor in~$F$.  Thus, replacing $v$ with~$w$ in~$F$ leads to an isomorphic graph, which remains witnessed.
  In a similar way we can replace other simplicial vertices of~$F$, leading to a claimed subgraph.
\end{proof}

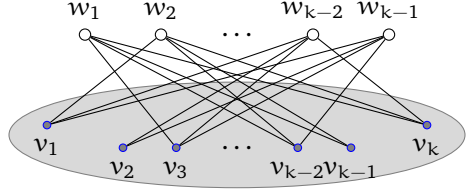
\begin{figure}[h]
  \centering \small
  \begin{tikzpicture}
    \draw[fill = gray!30, draw = gray] (2.5, 0.17) ellipse [x radius=3., y radius=0.7];
    \foreach \x/\l in {0/1, 5/k} {
      \node[filled vertex, "$v_{\l}$" below] (v\x) at (\x, 0.3) {};
    }
    \foreach[count=\i] \x/\l in {0/1, 1/2, 3/k-2, 4/k-1} 
    \node[empty vertex, "$w_{\l}$"] (w\i) at (\x+.5, 1.5) {};
    \foreach[count=\i] \x/\l in {1/2, 1.7/3, 3.3/k-2, 4/k-1} 
    \node[filled vertex, "$v_{\l}$" below] (v\i) at (\x, 0) {};

    \node at (2.5, 0) {$\ldots$};
    \node at (2.5, 1.5) {$\ldots$};

    \foreach \v/\list in {w1/{v2, v5, v4, v3}, w2/{v0, v5, v3, v4}, w3/{v0, v5, v1, v2}, w4/{v3, v0, v1, v2}} 
    \foreach \x in \list \draw (\x) -- (\v);
  \end{tikzpicture}
  \caption{The gadget.  Edges among vertices in the shadowed area, $\{v_{1}, v_{2}, \ldots, v_{k}\}$, are omitted for clarity.}
  \label{fig:gadget}
\end{figure}

For~$k \ge 2$, we define the gadget~$D_{k}$ as a subgraph of~$\overline{S_{i}}$ obtained by removing one simplicial vertex.
An example is illustrated in Figure~\ref{fig:gadget}, where the simplicial vertex adjacent to~$v_{2}, \ldots, v_{k-1}$ is removed.
The gadget~$D_{k}$ consists of~$2 k - 1$ vertices.  
The vertex set~$\{v_{1}, v_{2}, \ldots, v_{k}\}$ is a clique, while $\{w_{1}, w_{2}, \ldots, w_{k-1}\}$ is an independent set.
The vertex~$w_{i}, i = 1, \ldots, k-1$, is adjacent to~$\{v_{1}, \ldots, v_{i - 1}\}$ and~$\{v_{i + 2}, \ldots, v_{k}\}$.
For~$k \ge 2$, we define the graph~$S^{1}_{k}$ as follows.  We take a sun and the gadget~$D_{k}$, identify two degree-four vertices of the sun, chosen arbitrarily, and~$v_{1}$ and~$v_{k}$, and add all edges between~$\{v_{2}, \ldots, v_{k-1}\}$ and other vertices in the sun.
The graph~$S^{2}_{k}$ is defined similarly.
We take a rising sun and a gadget, identify the degree-five vertices of the rising sun and~$v_{1}$ and~$v_{k}$, and add all edges between~$\{v_{2}, \ldots, v_{k-1}\}$ and other vertices in the rising sun.
The graphs~$S^{1}_{k}$ and~$S^{2}_{k}$ for~$k = 2, 3, 4$ are shown in Figure~\ref{fig:split-non-helly-2}(c--e) and Figure~\ref{fig:split-non-helly-2}(f--h), respectively.
We leave it to the reader to verify that none of them is a circular-arc graph.\footnote{
To verify that all proper induced subgraphs of~$S^{1}_{k}$, $S^{2}_{k}, k \ge 2$ are Helly circular-arc graphs, one may draw circular-arc models for each of them.
An alternative approach is to use Theorem~\ref{thm:helly-main}.
If~$G = S^{1}_{k}$, then the graph~$G^{K}$ consists of the~$\dag$ graph of order~$k+4$, and~$w_{1}$, $\ldots$, $w_{k =1}$, where $w_{i}$ is adjacent to~$u_{0}, v_{i}$, and~$v_{i+1}$.  
The graph~$(G - w_{i})^{K}$ is isomorphic to the subgraph of~$(G)^{K} - w_{i}$ with the edge~$v_{i}v_{i+1}$ removed, while $(G - v)^{K\setminus \{v\}}$ for any~$v$ in the $\dag$ graph is isomorphic to the subgraph of~$(G)^{K} - v$.  Both are interval graphs, and thus all proper subgraphs of~$G$ are Helly circular-arc graphs.
It is similar for~$S^{2}_{k}$.
}

We are now ready to list minimal split graphs that are not Helly circular-arc graphs.
For the convenience of later references, we list the correspondence between them and minimal non-interval subgraphs of~$G^{K}$.

\begin{lemma}\label{lem:alternative}
  Let $G$ be a split graph with a split partition~$K\uplus S$.
  
  \begin{itemize}
  \item $G^{K}$ contains an induced sun if and only if $G$ contains an induced $\overline{S_{3}}$.
  \item $G^{K}$ contains an induced $C_{\ell}, \ell \ge 4$, 
    if and only if $G$ contains an induced $\overline{S_{\ell}}$.
  \item $G^{K}$ contains an induced and witnessed long claw if and only if $G$ contains an induced Figure~\ref{fig:long-claw-derived}.
  \item $G^{K}$ contains an induced and witnessed whipping top if and only if $G$ contains an induced Figure~\ref{fig:whipping-top-derived}.
  \item $G^{K}$ contains an induced and witnessed $\dag$ graph with $k$ vertices if and only if $G$ contains an induced $S^{1}_{k-4}$.
  \item $G^{K}$ contains an induced and witnessed $\ddag$ graph with $k$ vertices, with $k \ge 7$, if and only if $G$ contains an induced $S^{2}_{k-5}$.
  \end{itemize}
\end{lemma}
\begin{proof}
  Let $F$ be the vertex set of this subgraph of~$G^{K}$.  We use the labels in Figure~\ref{fig:non-interval} when $G^{K}[F]$ is chordal.

  Sun.  We may assume that all simplicial vertices of~$G^{K}[F]$ are from~$S$; otherwise, we can find an alternative induced sun with this property using Proposition~\ref{lem:h-maxclique} or \ref{lem:simplicial-iff-i-helly}.
  The only neighbors of~$x_{1}$, $x_{2}$, and~$x_{3}$ in~$G[F]$ are $u_{2}$, $v_{1}$, and~$u_{1}$, respectively.
  Thus, $G[F]$ is isomorphic to a net.

  
  Hole.  Let it be denoted as~$v_{1}v_{2} \cdots v_{\ell}$.
  Note that $v_{i}\in K$ for all $i = 1, \ldots, \ell$ by Proposition~\ref{lem:trivial}.
  For~$i = 1, \ldots, \ell$, we take a witness~$x_{i}$ for edge~$v_{i} v_{(i+1)\mod \ell}$.  In~$G$, we have $N(x_{i})\cap F = F\setminus \{v_{i}, v_{(i+1)\mod \ell}\}$.  The subgraph~$G[\{v_{1}, \ldots v_{\ell}, x_{1}, \ldots x_{\ell}\}]$ is isomorphic to~$\overline{S_{\ell}}$.

  In the rest, $G^{K}[F]$ is witnessed.
  We may assume that all simplicial vertices of~$G^{K}[F]$ are from~$S$; otherwise, we can find an alternative induced subgraph of~$G^{K}$ isomorphic to $G^{K}[F]$ with this property using Proposition~\ref{lem:simplicial-iff-i-helly}.
  
  Long claw.
  For~$i = 1, \ldots, 3$, take a witness~$w_{i}$ for edge~$v_{0} v_{i}$.
  Then $G[F\cup \{w_{1}, w_{2}, w_{3}\}]$ is isomorphic to Figure~\ref{fig:long-claw-derived}, where the degrees of~$x_{i}$ and~$w_{i}$ are three and two, respectively.

  Whipping top.  We take a witness~$w_{1}$ of~$\{v_{0}, v_{1}, v_{2}\}$ and a witness~$w_{2}$ of~$\{v_{0}, v_{2}, v_{3}\}$.
  Then $G[F\cup \{w_{1}, w_{2}, w_{3}\}]$ is isomorphic to Figure~\ref{fig:whipping-top-derived}, where $w_{1}$ and~$w_{2}$ have degree one, $x_{1}$ and~$x_{3}$ have degree two, while $x_{2}$ has degree three.
  
  $\dag$ graph.
  For~$i = 1, \ldots, p$, where $p = k-4$, take a witness~$w_{i}$ of~$\{u_{0}, v_{i}, v_{i+1}\}$.
  In~$G$, the set~$\{u_{0}, v_{1}, v_{p}, x_{1}, x_{2}, x_{3}\}$ induces a sun, while $\{v_{1}, \ldots, v_{p}, w_{1}, \ldots, w_{p-1}\}$ induces the gadget~$D_{p-1}$.  Thus, $G[F\cup \{w_{1}, \ldots, w_{p-1}\}]$ is isomorphic to~$S^{1}_{k-4}$.

  The $\ddag$ graph is similar.%
\end{proof}
  
Lemma~\ref{lem:alternative} implies the main result of this section, i.e., the first part of Theorem~\ref{thm:main}.

\begin{theorem}\label{thm:split-hcag}
  A split graph~$G$ is a Helly circular-arc graph
  if and only if it does not contain any induced copy of
Figure~\ref{fig:long-claw-derived}, Figure~\ref{fig:whipping-top-derived},
 $S^{1}_{k}$, $S^{2}_{k}$,
 or~$\overline{S_{k+1}}, k \ge 2$.
\end{theorem}
\begin{proof}
  Since none of the listed graphs is a Helly circular-arc graph, the necessity is straightforward.
  For sufficiency, we show that if $G$ not a Helly circular-arc graph, then it contains one of the list graphs as an induced subgraph.
 By Theorem~\ref{thm:helly-main}, $G^{K}$ is not an interval graph, and it contains an induced non-interval subgraph~$F$.
 If $G^{K}$ contains an induced sun, then $G$ contains $\overline{S_{3}}$.
 Otherwise, $F$ is witnessed by Proposition~\ref{lem:h-maxclique}, and the statement follows from Lemma~\ref{lem:alternative}.
\end{proof}

\section{Split graphs that are not circular-arc graphs}
\label{sec:cag}

Let $G$ be a split graph with no universal vertices, and~$K\uplus S$ a split partition of~$G$.
We may assume that there are no \emph{true twins}, two vertices~$x$ and~$y$ with the same closed neighborhood.
If there are two true twins, then $G$ is a circular-arc graph if and only if $G - y$ is a circular-arc graph.
We fix a simplicial vertex~$s$ of~$G$ and we use $H$ to denote $G^{N[s]}$.  We further partition $K$ into
\[
  K_{s} = N_{G}(s) \text{ and } K_{o} = K\setminus N_{G}(s).
\]
The structure of~$H$ can be summarized as follows.
The edges among vertices in~$V(G)\setminus N_{G}[s]$ are the same as in~$G$; in particular, $S\setminus\{s\}$ remains an independent set, and~$K_{o}$ remains a clique.
A pair of vertices~$v\in K_{s}$ and~$x\in S\setminus \{s\}$ are adjacent in~$H$ if and only if they are not adjacent in~$G$.
A pair of vertices~$u, v\in K_{s}$ are adjacent in~$H$ if and only if there exists a vertex adjacent to neither of them, i.e., $N_{G}(u)\cup N_{G}(v)\ne V(G)$.
A pair of vertices~$u\in K_{s}$ and~$v\in K_{o}$ are adjacent in~$H$ if and only if there exists a vertex adjacent to~$v$ but not $u$ in~$G$.  Note that these witnesses must be from~$S$, and we do not need a witness for an edge between two vertices in~$K_{o}$.

\subsection{Forbidden configurations}

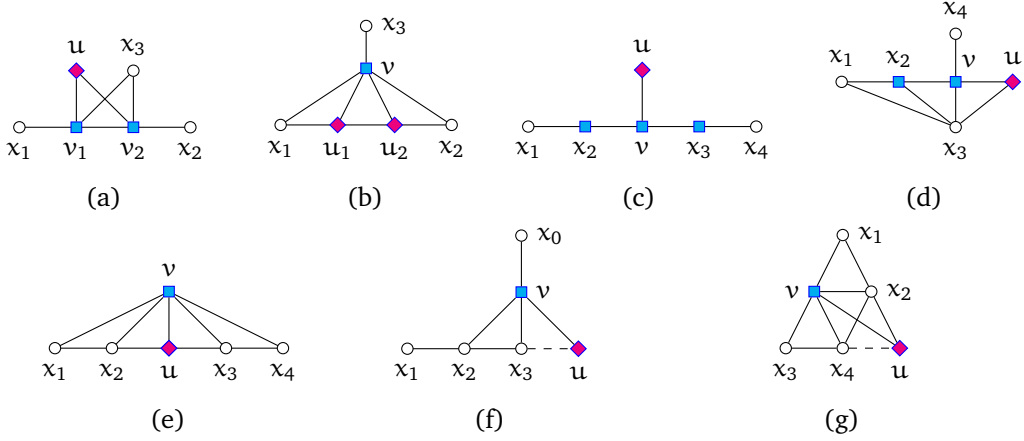
\begin{figure}[h]
  \centering \small
\begin{subfigure}[b]{0.2\linewidth}
    \centering
    \begin{tikzpicture}[scale=.75]
      \draw (1, 0) -- (4, 0);
      \foreach[count=\x from 2] \t/\l/\p in {b-/u/above, empty /x_{3}/below}  \node[\t vertex, "$\l$"] (\t) at (\x, 1) {};
      
      \foreach[count=\i] \t/\l in {empty /x_{1}, a-/v_{1}, a-/v_{2}, empty /x_{2}} {
        \node[\t vertex, "$\l$" below] (u\i) at (\i, 0) {};
      }
      \foreach \i in {2, 3} \draw (b-) -- (u\i) -- (empty );
    \end{tikzpicture}
    \caption{}
    \label{fig:configuration-7}
  \end{subfigure}
  \,
  \begin{subfigure}[b]{0.2\linewidth}
    \centering
    \begin{tikzpicture}[scale=.75]
      \draw (1, 0) -- (4, 0);
      \node[a-vertex, "$v$" right] (a) at (2.5, 1) {};
      \draw (a) -- ++(0, .75) node[empty vertex, "$x_{3}$" right] {};

      \foreach[count=\i] \t/\l in {empty /x_{1}, b-/u_{1}, b-/u_{2}, empty /x_{2}} {
        \draw (a) -- (\i, 0) node[\t vertex, "$\l$" below] (u\i) {};
      }
    \end{tikzpicture}
    \caption{}
    \label{fig:configuration-8}
  \end{subfigure}
  \,
  \begin{subfigure}[b]{0.22\linewidth}
    \centering
    \begin{tikzpicture}[scale=.75]
      \draw (-2.,0) node[empty vertex, "$x_1$" below] (x1) {} -- (2.,0) node[empty vertex, "$x_4$" below] (x2) {};
      \foreach[count=\i] \l in {x_2, v, x_3}
      \node[a-vertex, "$\l$" below] (v\i) at ({\i*1-2}, 0) {};
      \draw (90:1) node[b-vertex, "$u$"] (u) {} -- (v2);
    \end{tikzpicture}
    \caption{}
    \label{fig:configuration-1}
  \end{subfigure}
  \,
  \begin{subfigure}[b]{0.22\linewidth}
    \centering
    \begin{tikzpicture}[scale=.75]
      \draw (1, 1) -- (4, 1);

       \node[empty vertex, "$x_{3}$" below] (v4) at (3, 0.2) {};
      \foreach[count=\i] \t/\v/\x in {empty /x_1/1, a-/x_2/2, a-/{\quad v}/3.5, b-/u/5} {
        \draw (v4) -- (\i, 1) node[\t vertex, "$\v$"] (u\i) {};
      }

      \draw (u3) -- ++(0, .85) node[empty vertex, "$x_{4}$"] (x2) {};
     \end{tikzpicture}
    \caption{}
    \label{fig:configuration-2}
  \end{subfigure}

  \begin{subfigure}[b]{0.25\linewidth}
    \centering
    \begin{tikzpicture}[scale=.75]
      \draw (-2.,0) -- (2.,0);
      \node[a-vertex, "$v$"] (a) at (0, 1.) {};
      \draw (a) -- (0, 0) node[b-vertex, "$u$" below] (u) {};
      \node[b-vertex] (c) at (0, 0) {};
      \foreach[count=\j] \x in {-2, -1, 1, 2} {
        \draw (a) -- (\x, 0) node[empty vertex, "$x_\j$" below] (v\j) {};
      }
     \end{tikzpicture}
    \caption{}
    \label{fig:configuration-3}
  \end{subfigure}
  \,
  \begin{subfigure}[b]{0.25\linewidth}
    \centering
    \begin{tikzpicture}[scale=.75]
      \node[a-vertex, "$v$" right] (x) at (3., 1) {};
      \node[empty vertex, "$x_{0}$" right] (u0) at (3., 2) {};

      \draw (1, 0) -- (3, 0);
      \foreach[count=\i] \t/\v in {empty /x_{1}, empty /x_{2}, empty /x_{3}, b-/u} {
        \node[\t vertex, "$\v$" below] (u\i) at (\i, 0) {};
      }
      \draw[dashed] (u3) -- (u4);
      \foreach \i in {0, 2, 3, 4} \draw (x) -- (u\i);
    \end{tikzpicture}
    \caption{}
    \label{fig:configuration-4}
  \end{subfigure}
  \,
  \begin{subfigure}[b]{0.3\linewidth}
    \centering
    \begin{tikzpicture}[scale=.75]
      \foreach \i in {1, 2} {
        \node[empty vertex, "$x_\i$" right] (x\i) at (1.5+.5*\i, 3-\i) {};
      }
      \foreach \i in {3, 4} 
      \node[empty vertex, "$x_\i$" below] (x\i) at (\i-2, 0) {};
      \node[a-vertex, "$v$" left] (v) at (1.5, 1) {};
      \node[b-vertex, "$u$" below] (x5) at (3., 0) {};

      \foreach \i in {1, 4, 5} \draw (x2) -- (x\i);
      \foreach \i in {1, ..., 5} \draw (v) -- (x\i);
      \draw (x3) -- (x4) (x4) edge[dashed] (x5);
    \end{tikzpicture}
    \caption{}
    \label{fig:configuration-6}
  \end{subfigure}
  \caption{Forbidden configurations of~$H$ that are interval graphs.
    The square and rhombus nodes are from~$K_{s}$ and~$K_{o}$, respectively, while round nodes are not specified.
    A dashed line indicates a path of an arbitrary length.  There are at least six vertices in (f).
    }
  \label{fig:non-cag-to-interval}
\end{figure}

Unlike Theorem~\ref{thm:helly-main}, $H$ being an interval graph is insufficient for~$G$ to be a circular-arc graph.
For example, when $G$ is a net$^\star$ and~$s$ is a degree-one vertex of~$G$, the graph~$H$ is an interval graph, consisting of the graph in Figure~\ref{fig:configuration-8} and a universal vertex.
To satisfy condition \eqref{eq:1}, $H$ cannot contain any graph in Figure~\ref{fig:non-cag-to-interval}.
All these graphs are proper subgraphs of minimal non-interval graphs (Figure~\ref{fig:non-interval}), and in each of them, there are two or more vertices earmarked from~$K_{s}$ and~$K_{o}$.
We say that $H$ \emph{contains an annotated copy} of a graph~$F$ if there exists an isomorphism between~$F$ and an induced subgraph of~$H$ that preserves the vertex partition; i.e., a vertex of~$K_{s}$ or~$K_{o}$ must be mapped to a vertex in the same set.

\begin{proposition}\label{lem:necessity-forbidden-configurations}
  If $H$ contains an annotated copy of any graph in Figure~\ref{fig:non-cag-to-interval}, then it does not satisfies condition~\eqref{eq:1}.
\end{proposition}
\begin{proof}
  In any interval model of Figure~\ref{fig:configuration-7}, at least one of the intervals for~$v_{1}$ and~$v_{2}$ properly contains the interval for~$u$.
  In any interval model of Figure~\ref{fig:configuration-8}, the interval for~$v$ properly contains at least one of the intervals for~$u_{1}$ and~$u_{2}$.
  In any interval model of the other graphs in Figure~\ref{fig:non-cag-to-interval}, the interval for~$v$ properly contains the interval for~$u$.
\end{proof}

The rest of this subsection is devoted to proving the other direction of Proposition~\ref{lem:necessity-forbidden-configurations}; i.e., the exclusion of the graphs in Figure~\ref{fig:non-cag-to-interval} ensures condition \eqref{eq:1}.
Condition~\eqref{eq:1} is better understood through clique paths of the graph~$H$.
Let us take an arbitrary interval model of~$H$.  If for each of~$n$ (not necessarily distinct) left endpoints of the $n$ intervals, we take the set of vertices whose intervals contain this point, then we end with $n$ cliques.  We leave it to the reader to verify that they include all the maximal cliques of~$H$.  If we list the distinct maximal cliques from left to right, sorted by the endpoints that we use to define these cliques, then we can see that for any~$v\in V(H)$, the maximal cliques of~$H$ containing $v$ appear consecutively.  We say that such a linear arrangement of maximal cliques is a \emph{clique path} of~$H$.  On the other hand, given a clique path~$\langle K_{1}, K_{2}, \ldots, K_{\ell}\rangle$ for an interval graph~$H$ with $\ell$ maximal cliques, for each vertex~$v$ we can define an interval~$[\lk(v), \rk(v)]$, where $\lk(v)$ and~$\rk(v)$ are the indices of the first and, respectively, last maximal cliques containing $v$.  One may easily see that they define an interval model for~$H$; see, e.g., Figure~\ref{fig:clique-path-interval-model}.  Therefore, clique paths and interval models
are interchangeable, and when we illustrate clique paths, we always use the way in Figure~\ref{fig:clique-path-interval-model}.

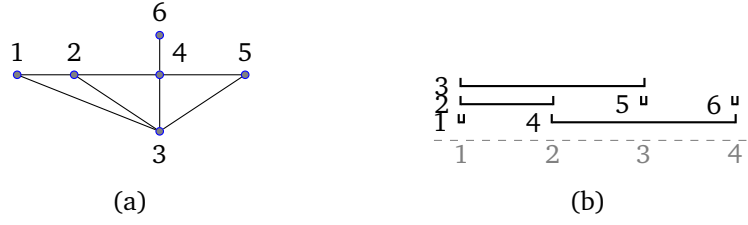
\begin{figure}[h]
  \centering\small
  \begin{subfigure}[b]{.35\linewidth}
    \centering
    \begin{tikzpicture}[scale=.75]
      \draw (1, 1) -- (5, 1);
      \node[filled vertex, "$3$" below] (v4) at (3.5, 0) {};
      \foreach[count=\i] \v/\x in {1/1, 2/2, /3.5, 5/5} {
        \draw (v4) -- (\x, 1) node[filled vertex, "$\v$"] (u\i) {};
      }
      \node[above right = 0mm of u3] {$4$};
      \draw (u3) -- ++(0, .7) node[filled vertex, "$6$"] (x2) {};
    \end{tikzpicture}
    \caption{}
  \end{subfigure}  
  \quad
  \begin{subfigure}[b]{.35\linewidth}
    \centering
    \begin{tikzpicture}[scale=1.2]
      \begin{scope}[every path/.style={{|[right]}-{|[left]}}]
        \foreach \x in {1, ..., 4} \node[gray] at (\x-1, -.15) {$\x$};
        \foreach[count=\i] \l/\r/\y/\c in {-.02/.02/1/, 0/1/2/, 0/2/3/, 1/3/1/, 1.98/2.02/2/, 2.98/3.02/2/}{
          \draw[\c, thick] (\l-.02, \y/5) node[left] {$\i$} to (\r+.02, \y/5);
        }
      \end{scope}
      \draw[dashed, gray] (-.3, 0) -- (3.2, 0);
    \end{tikzpicture}
    \caption{}
  \end{subfigure}  
  \caption{An interval graph and its clique path represented as an interval model.}
  \label{fig:clique-path-interval-model}
\end{figure}

Condition~\eqref{eq:1} can be translated into the language of clique paths as: at least one of~$\lk(u) = \lk(v)$ and~$\rk(u) = \rk(v)$ for all pairs~$v\in K_{s}$ and~$u\in K_{o}$.

\begin{proposition}\label{lem:clique-path}
  The graph~$H$ satisfies \eqref{eq:1} if and only if there exists a clique path of~$H$ such that
  \begin{equation}
    \label{eq:2}
    N_{H}(v) \cap K_{o} \subseteq K_{\lk(v)}\cup K_{\rk(v)} \text{ for every } v\in K_{s}.
  \end{equation}
\end{proposition}
\begin{proof}
  For the sufficiency, we show that $H$ satisfies \eqref{eq:1} with the following interval model:
  \[
    I(x) =
    \begin{cases}
      [\lk(x) - {1\over 3}, \rk(x) + {1\over 3}] & \text{if } x\in K_{o}\cup\{s\}, \\
      [\lk(x), \rk(x)] & \text{otherwise.}
    \end{cases}
  \]
  The vertex~$s$ is universal, and~$I(s)$ contains all other intervals.
  Now consider a pair of vertices~$v\in K_{s}$ and~$u\in N_{H}(v)\setminus K_{s}$.
  If~$u\in S$, then $uv\not\in E(G)$ and~$I(u)\subset I(v)$ by Proposition~\ref{lem:trivial}.
  Otherwise, $uv\in E(G)$ and~$I(u)\setminus I(v)$ contains at least one of~$[\lk(x) - {1\over 3}, \lk(x)]$ and~$[\rk(x), \rk(x) + {1\over 3}]$.
  
  For necessity, suppose $H$ satisfies \eqref{eq:1}.
  We use the clique path obtained from an interval model specified in \eqref{eq:1}.
  Let $u$ be a vertex in~$N_{H}(v)\cap K_{o}$.
  Since $u$ and~$v$ are adjacent in~$G$, it follows $I(u)\not\subseteq I(v)$.  If the left endpoint of~$u$ is not in~$I(v)$, then $u\in K_{\lk(v)}$.
  Otherwise, the right endpoint of~$u$ is not in~$I(v)$, and~$u\in K_{\rk(v)}$.
\end{proof}

We need to introduce a few definitions and recall known facts.
A subset~$M$ of vertices forms a \emph{module} of~$H$ if for any pair of vertices~$u,v \in M$, a vertex~$x\in V(H)\setminus M$ is adjacent to~$u$ if and only if it is adjacent to~$v$ as well; e.g., $\{x_{1}, u_{1}, u_{2}, x_{2}\}$ of the graph in Figure~\ref{fig:configuration-8}.
The set~$V(H)$ and all singleton vertex sets are modules, called \emph{trivial}.
We say that a graph is \emph{quasi-prime} if every nontrivial module is a clique.

\begin{theorem}[\cite{hsu-95-independent-set-cag, cao-21-interval-recognition}]
  \label{thm:unique-clique-path}
  An interval graph that is quasi-prime has a unique clique path, up to full reversal.
\end{theorem}

The following is straightforward.

\begin{proposition}\label{lem:modules-interval}
  Let $\langle K_{1}, \ldots, K_{\ell}\rangle$ be a clique path of an interval graph~$H$.
  If $H$ is quasi-prime, then there cannot be distinct indices~$p, q$ with $[p, q]\subset [1, \ell]$ such that
both $K_{p - 1}\cap K_{p}\setminus K_{q}$ and~$K_{q}\cap K_{q + 1}\setminus K_{p}$ are empty, where $K_{0} = K_{\ell+1} = \emptyset$.
\end{proposition}
\begin{proof}
  If $K_{p - 1}\cap K_{p}\setminus K_{q}$ and~$K_{q}\cap K_{q + 1}\setminus K_{p}$ are both empty, then
\[
  M = \bigcup_{i=p}^q K_{i}\setminus (K_{p}\cap K_{q})
\]
is a module of~$G$.  For every vertex~$x\in M$, it holds $N(x)\setminus M = K_{p}\cap K_{q}$.  This module is not a clique because it contains a vertex in~$K_{p}\setminus K_{p+1}\subseteq K_{p}\setminus K_{q}$ and a vertex in~$K_{q}\setminus K_{q-1}\subseteq K_{q}\setminus K_{p}$, which cannot be adjacent.  
\end{proof}

Finally, we need a result of Gimbel~\cite{gimbel-88-end-vertices} on interval graphs.

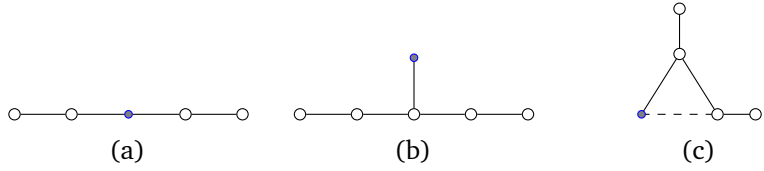
\begin{figure}[h]
  \centering \small
  \begin{subfigure}[b]{0.22\linewidth}
    \centering
    \begin{tikzpicture}[scale=.75]
      \node[filled vertex] (c) at (0, 0) {};
      \foreach \i in {1, 3} {
        \node[empty vertex] (u\i) at ({90*(\i-1)}:2) {};
        \node[empty vertex] (v\i) at ({90*(\i-1)}:1) {};
        \draw (u\i) -- (v\i) -- (c);
      }
    \end{tikzpicture}
    \caption{}
  \end{subfigure}
  \,
  \begin{subfigure}[b]{0.22\linewidth}
    \centering
    \begin{tikzpicture}[scale=.75]
      \node[empty vertex] (c) at (0, 0) {};
      \foreach \i in {1, 3} {
        \node[empty vertex] (u\i) at ({90*(\i-1)}:2) {};
        \node[empty vertex] (v\i) at ({90*(\i-1)}:1) {};
        \draw (u\i) -- (v\i) -- (c);
      }
      \draw (90:1) node[filled vertex] {} -- (c);
    \end{tikzpicture}
    \caption{}
  \end{subfigure}
  \,
  \begin{subfigure}[b]{0.22\linewidth}
    \centering
    \begin{tikzpicture}[yscale=.8]
      \draw[dashed] (0, 0) -- (1., 0);
      \draw (0.5, 1.75) node[empty vertex] {} -- (0.5, 1) node[empty vertex] (c) {} -- (0, 0) node[filled vertex] {};
      \foreach[count =\i] \x/\t in {1/empty } {
        \node[\t vertex] (v\i) at ({1. * \x}, 0) {};
        \draw (v\i) -- (c);
      }
      \draw (v1) -- (1.5, 0) node[empty vertex] {};
    \end{tikzpicture}
    \caption{}
  \end{subfigure}
  \caption{The solid node cannot be in an end clique in any clique path.  A dashed line indicates a path of an arbitrary length.}
  \label{fig:non-end-interval}
\end{figure}

\begin{theorem}[\cite{gimbel-88-end-vertices}]
  \label{thm:end-interval}
  Let $H$ be an interval graph, and~$v$ a vertex.
  There is a clique path of~$H$ in which $v$ is in the first or last clique
  if and only if $v$ is not the highlighted vertex of an induced subgraph in Figure~\ref{fig:non-end-interval}.
\end{theorem}

We prove the main result in two steps.  First, we deal with the case that $H$ is quasi-prime.

\begin{lemma}\label{lem:quasi-prime}
  Let $H'$ be an induced subgraph of~$H$ that is quasi-prime.
  If $H$ is an interval graph and does not contain an annotated copy of any graph in Figure~\ref{fig:non-cag-to-interval}, then any clique path of~$H'$ satisfies condition~\eqref{eq:2}.
\end{lemma}
\begin{proof}
  Let $\langle K_{1}, \ldots, K_{\ell}\rangle$ be a clique path of~$H'$.
  We show that if 
  there are $u\in K_{o}$ and~$v\in K_{s}$ such that
  \[
    \lk(v) < \lk(u) \le \rk(u) < \rk(v),
  \]
  then $H'$ contains an annotated copy of a graph in Figure~\ref{fig:non-cag-to-interval}.
  Since a clique path of~$H'$ is either  $\langle K_{1}, \ldots, K_{\ell}\rangle$ or its reversal (Theorem~\ref{thm:unique-clique-path}), $u$ cannot be in an end clique in either of them.
  By Theorem~\ref{thm:end-interval}, $u$ is the highlighted vertex of an induced subgraph in Figure~\ref{fig:non-end-interval}.
  If $\lk(v) = 1$ \textit{and}~$\rk(v) = \ell$, then $H'$ contains an annotated copy of Figure~\ref{fig:configuration-6} (when the induced subgraph is Figure~\ref{fig:non-end-interval}c), or Figure~\ref{fig:configuration-1} (when the induced subgraph is Figure~\ref{fig:non-end-interval}b and the neighbor of~$u$ is in~$K_{s}$) or Figure~\ref{fig:configuration-3} (otherwise).
  Henceforth, we assume that $v$ is not universal in~$H'$.
  
For notational convenience, we introduce empty sets~$K_{0}$ and~$K_{\ell+1}$ as sentinels.
Since $H'$ is quasi-prime and~$v$ is not universal in~$H'$, at least one of 
\begin{align*}
  L =& K_{\lk(v) - 1}\cap K_{\lk(v)}\setminus K_{\rk(v)},
  \\
  R =& K_{\rk(v)}\cap K_{\rk(v) + 1}\setminus K_{\lk(v)}
\end{align*}
is nonempty by Proposition~\ref{lem:modules-interval}.
Note that $\lk(v) > 1$ if $L\ne\emptyset$, and~$\rk(v) < \ell$ if $R\ne\emptyset$.
  Since $K_{i}$, $i = 1, \ldots, \ell$, is a maximal clique of $H'$, neither $K_{i}\setminus K_{i+1}$ nor~$K_{i}\setminus K_{i-1}$ can be empty.
  
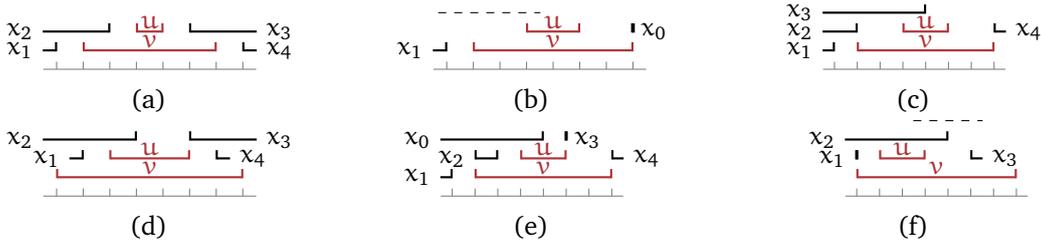
\begin{figure}[h]
  \centering \small
  \begin{subfigure}[b]{.3\linewidth}
    \centering
    \begin{tikzpicture}[xscale=.35]
      \foreach \l/\r/\y/\c in {4/5/2/u, 2/7/1/v}{
        \draw[{|[right]}-{|[left]}, thick, Sepia] (\l-.02, \y/4) to node[midway, yshift=3.5pt] {$\c$} (\r+.02, \y/4);
      }
      \foreach[count=\i] \right/\span/\y in {1/.5/1, 3/2.5/2}{
        \draw[-{|[left]}, thick] (\right-\span-.02, \y/4) node[left] {$x_\i$} to (\right+.02, \y/4);
      }
      \foreach[count=\i from 3] \left/\span/\y in {6/2.5/2, 8/.5/1}{
        \draw[{|[right]}-, thick] (\left-.02, \y/4) to (\left+\span+.02, \y/4) node[right] {$x_\i$};
      }
      \draw[gray] (.5, 0) -- (8.5, 0);
      \foreach \x in {1, ..., 8}
      \draw[dashed, gray] (\x, 0) -- ++(0, .1);
    \end{tikzpicture}
    \caption{}
    \label{fig:clique-path-1}
  \end{subfigure}  
  \,
  \begin{subfigure}[b]{.3\linewidth}
    \centering
    \begin{tikzpicture}[xscale=.35]
      \foreach \l/\r/\y/\c in {4/6/2/u, 2/8/1/v}{
        \draw[{|[right]}-{|[left]}, thick, Sepia] (\l-.02, \y/4) to node[midway, yshift=3.5pt] {$\c$} (\r+.02, \y/4);
      }
      \foreach[count=\i from 0] \l/\r/\y in {7.95/8.05/2}{
        \draw[{|[right]}-{|[left]}, thick] (\l-.02, \y/4) node[right] {$x_\i$} to (\r+.02, \y/4);
      }
      \foreach[count=\i] \right/\span/\y in {1/.5/1}{
        \draw[-{|[left]}, thick] (\right-\span-.02, \y/4) node[left] {$x_\i$} to (\right+.02, \y/4);
      }
      \foreach \l/\r/\y in {.7/4.5/3}{
        \draw[dashed] (\l-.02, \y/4) to (\r+.02, \y/4);
      }
      \draw[gray] (.5, 0) -- (8.5, 0);
      \foreach \x in {1, ..., 8}
      \draw[dashed, gray] (\x, 0) -- ++(0, .1);
    \end{tikzpicture}
    \caption{}
    \label{fig:clique-path-2}
  \end{subfigure}  
  \,
  \begin{subfigure}[b]{.3\linewidth}
    \centering
    \begin{tikzpicture}[xscale=.3]
      \foreach \l/\r/\y/\c in {4/6/2/u, 2/8/1/v}{
        \draw[{|[right]}-{|[left]}, thick, Sepia] (\l-.02, \y/4) to node[midway, yshift=3.5pt] {$\c$} (\r+.02, \y/4);
      }
      \foreach[count=\i] \right/\span/\y in {1/.5/1, 2/1.5/2, 5/4.5/3} 
      \draw[-{|[left]}, thick] (\right-\span-.02, \y/4) node[left] {$x_\i$} to (\right+.02, \y/4);
      \foreach[count=\i from 4] \left/\span/\y in {8/0.5/2}
        \draw[{|[right]}-, thick] (\left-.02, \y/4) to (\left+\span+.02, \y/4) node[right] {$x_\i$};
      \draw[gray] (.5, 0) -- (8.5, 0);
      \foreach \x in {1, ..., 8}
      \draw[dashed, gray] (\x, 0) -- ++(0, .1);
    \end{tikzpicture}
    \caption{}
    \label{fig:clique-path-3}
  \end{subfigure}  

  \begin{subfigure}[b]{.3\linewidth}
    \centering
    \begin{tikzpicture}[xscale=.35]
      \foreach \l/\r/\y/\c in {3/6/2/u, 1/8/1/v}{
        \draw[{|[right]}-{|[left]}, thick, Sepia] (\l-.02, \y/4) to node[midway, yshift=3.5pt] {$\c$} (\r+.02, \y/4);
      }

      \foreach[count=\i] \right/\span/\y in {2/.5/2, 4/3.5/3}{
        \draw[-{|[left]}, thick] (\right-\span-.02, \y/4) node[left] {$x_\i$} to (\right+.02, \y/4);
      }
      \foreach[count=\i from 3] \left/\span/\y in {6/2.5/3, 7/0.5/2}{
        \draw[{|[right]}-, thick] (\left-.02, \y/4) to (\left+\span+.02, \y/4) node[right] {$x_\i$};
      }
      \draw[gray] (.5, 0) -- (8.5, 0);
      \foreach \x in {1, ..., 8}
      \draw[dashed, gray] (\x, 0) -- ++(0, .1);

    \end{tikzpicture}
    \caption{}
    \label{fig:clique-path-4}
  \end{subfigure}  
  \,
  \begin{subfigure}[b]{.3\linewidth}
    \centering
    \begin{tikzpicture}[xscale=.3]
      \foreach \l/\r/\y/\c in {4/6/2/u, 2/8/1/v}{
        \draw[{|[right]}-{|[left]}, thick, Sepia] (\l-.02, \y/4) to node[midway, yshift=3.5pt] {$\c$} (\r+.02, \y/4);
      }
      \foreach[count=\i from 0] \right/\span/\y in {5/4.5/3, 1/.5/1} 
      \draw[-{|[left]}, thick] (\right-\span-.02, \y/4) node[left] {$x_\i$} to (\right+.02, \y/4);
      \foreach[count=\i from 4] \left/\span/\y in {8/.5/2}
        \draw[{|[right]}-, thick] (\left-.02, \y/4) to (\left+\span+.02, \y/4) node[right] {$x_\i$};

      \foreach[count=\i from 2] \l/\r/\y/\p in {2/3/2/left, 5.95/6.05/3/right}{
        \draw[{|[right]}-{|[left]}, thick] (\l-.02, \y/4) node[\p] {$x_\i$} to (\r+.02, \y/4);
      }

      \draw[gray] (.5, 0) -- (8.5, 0);
      \foreach \x in {1, ..., 8}
      \draw[dashed, gray] (\x, 0) -- ++(0, .1);
    \end{tikzpicture}
    \caption{}
    \label{fig:clique-path-5}
  \end{subfigure}  
  \,
  \begin{subfigure}[b]{.3\linewidth}
    \centering
    \begin{tikzpicture}[xscale=.3]
      \foreach \l/\r/\y/\c in {2/4/2/u, 1/8/1/v}{
        \draw[{|[right]}-{|[left]}, thick, Sepia] (\l-.02, \y/4) to node[midway, yshift=3.5pt] {$\c$} (\r+.02, \y/4);
      }

      \foreach[count=\i] \l/\r/\y/\p in {0.95/1.05/2/left}{
        \draw[{|[right]}-{|[left]}, thick] (\l-.02, \y/4) node[\p] {$x_\i$} to (\r+.02, \y/4);
      }

      \foreach[count=\i from 2] \right/\span/\y in {5/4.5/3} 
      \draw[-{|[left]}, thick] (\right-\span-.02, \y/4) node[left] {$x_\i$} to (\right+.02, \y/4);
      \foreach[count=\i from 3] \left/\span/\y in {6/0.5/2}
        \draw[{|[right]}-, thick] (\left-.02, \y/4) to (\left+\span+.02, \y/4) node[right] {$x_\i$};
  
      \foreach \l/\r/\y in {3.5/6.5/4}{
        \draw[dashed] (\l-.02, \y/4) to (\r+.02, \y/4);
      }

      \draw[gray] (.5, 0) -- (8.5, 0);
      \foreach \x in {1, ..., 8}
      \draw[dashed, gray] (\x, 0) -- ++(0, .1);
    \end{tikzpicture}
    \caption{}
    \label{fig:clique-path-6}
  \end{subfigure}  
  \caption{Clique paths used in the proof of Lemma~\ref{lem:quasi-prime}.}
  \label{fig:clique-path}
\end{figure}

  Case 1, $L\cup R$ is disjoint from~$N(u)$.
  If neither $L$ nor~$R$ is empty, then $H$ contains an annotated copy of Figure~\ref{fig:configuration-1}, where
    $x_{1}\in K_{\lk(v) - 1}\setminus K_{\lk(v)}$, $x_{2}\in L$,
  and
    $x_{3}\in R$, $x_{4}\in K_{\rk(v) + 1}\setminus K_{\rk(v)}$ (see Figure~\ref{fig:clique-path-1}).
    Hence, we assume without loss of generality that $L\ne\emptyset$ and~$R=\emptyset$.
    We argue that $H$ contains an annotated copy of Figure~\ref{fig:configuration-4}.
   We take a vertex~$x_{0}\in K_{\rk(v)}\setminus K_{\rk(v) - 1}$ and a vertex~$x_{1}\in K_{\lk(v) - 1}\setminus K_{\lk(v)}$; see Figure~\ref{fig:clique-path-2}.
   By Proposition~\ref{lem:modules-interval}, $K_{i-1}\cap K_{i}\setminus K_{\rk(v)}\ne \emptyset$ for all~$i = \lk(v), \ldots, \lk(u)$.
   Thus, $x_{1}$ and~$u$ are in the same component of~$H' - K_{\rk(v)}$.
   We find a shortest \stpath{x_{1}}{u} in~$H' - K_{\rk(v)}$.
   Note that the length of this path is at least two (because $L\cap N(u) = \emptyset$), and its internal vertices are adjacent to~$v$ but not $x_{0}$.

  Case 2, $u$ has at least one neighbor in~$L\cup R$.
  We may assume without loss of generality that $L\cap N(u)\ne\emptyset$; it is symmetric if $R\cap N(u)\ne\emptyset$.
  If $L \not\subseteq N(u)$, then $H$ contains an annotated copy of Figure~\ref{fig:configuration-2}, where $x_{1}\in K_{\lk(v) - 1}\setminus K_{\lk(v)}$, $x_{2}\in L \setminus N(u)$, $x_{3}\in L\cap N(u)$, and~$x_{4}\in K_{\rk(v)}\setminus K_{\rk(v) - 1}$ (see Figure~\ref{fig:clique-path-3}).
  Henceforth, $L$ is a nonempty subset of~$N(u)$.

  Case 2.1, $L\not\subseteq K_{\rk(u)}$.  
  If $K_{\rk(u)}\cap K_{\rk(u)+1}\not\subseteq K_{\rk(u) - 1}$,
  then $H$ contains an annotated copy of Figure~\ref{fig:configuration-3}, where $x_{1}\in K_{\lk(u)-1}\setminus K_{\lk(u)}$, $x_{2}\in L\setminus K_{\rk(u)}$, $x_{3}\in K_{\rk(u)}\cap K_{\rk(u)+1}\setminus K_{\rk(u) - 1}$, and~$x_{4}\in K_{\rk(u)+1}\setminus K_{\rk(u)}$ (see Figure~\ref{fig:clique-path-4}).
  Otherwise, we find a vertex~$x_{1}\in K_{\lk(v) - 1}\setminus K_{\lk(v)}$, a vertex~$x_{2}\in K_{\lk(u) - 1}\setminus K_{\lk(u)}$, a vertex~$x_{3}\in K_{\rk(u)}\setminus K_{\rk(u) - 1}$, and a vertex~$x_{4}\in K_{\rk(v)}\setminus K_{\rk(v) - 1}$; see Figure~\ref{fig:clique-path-5}.
    Since $L\subseteq N(u)$, the vertex~$x_{2}$ cannot be in~$K_{\lk(v) - 1}$; since
    $K_{\rk(u)}\cap K_{\rk(u)+1}\setminus K_{\rk(u) - 1}= \emptyset$,
    the vertex~$x_{3}$ cannot be in~$K_{\rk(u) + 1}$.
    Note that $x_{0}\in K$ by Proposition~\ref{lem:trivial}.
    Hence, $H$ contains an annotated copy of Figure~\ref{fig:configuration-7}
    if $x_{0}\in K_{s}$, or Figure~\ref{fig:configuration-8} otherwise, with vertex set~$\{u, v, x_{0}, x_{1}, x_{2}, x_{4}\}$ and~$\{u, v, x_{0}, x_{2}, x_{3}, x_{4}\}$, respectively.

    Case 2.2, $L\subseteq K_{\rk(u)}$.
    We argue that $H$ contains an annotated copy of Figure~\ref{fig:configuration-6}.
    We take a vertex~$x_{1}\in K_{\lk(v)}\setminus K_{\lk(v) + 1}$, a vertex~$x_{2}$ from~$L$ such that $\rk(x_{2})$ is minimized, and a vertex~$x_{3}\in K_{\rk(x_{2}) + 1}\setminus K_{\rk(x_{2})}$.  See Figure~\ref{fig:clique-path-6}.
    By Proposition~\ref{lem:modules-interval}, $K_{i}\cap K_{i+1}\setminus K_{\lk(v)}\ne \emptyset$ for all~$i = \rk(u), \rk(u)+1, \ldots, \rk(x_{2})$.
   Thus, $u$ and~$x_{3}$ are in the same component of~$H' - K_{\lk(v)}$.
   We find a shortest \stpath{u}{x_{3}} in~$H' - K_{\lk(v)}$.
   Note that all internal vertices of this path, which is nontrivial because $x_{3}u\not\in E(H')$, are adjacent to~$x_{2}$ but not~$x_{1}$.
\end{proof}

We are now ready for the general case.

\begin{lemma}\label{lem:split-non-cag-but-interval}
  A split graph~$G$ is a circular-arc graph if and only if $H$ is an interval graph and does not contain an annotated copy of any graph in Figure~\ref{fig:non-cag-to-interval}.
\end{lemma}
\begin{proof}
  The necessity is given in Proposition~\ref{lem:necessity-forbidden-configurations}, and the proof is focused on the sufficiency.  
  By Theorem~\ref{thm:correlation} and Proposition~\ref{lem:clique-path}, it suffices to construct a clique path of~$H$ satisfying condition~\eqref{eq:2}.
  We take an induced subgraph~$H'$ of~$H$ by applying the following operations exhaustively (the application of one of them may re-enable the other).
  Initially, $H' = H$.  Since $K_{o}$ is a clique, $K_{o}\cap V(H')$ must reside in a single component of~$H'$.
  \begin{itemize}
  \item If $H'$ is not connected, remove components disjoint from~$K_{o}$.
  \item Remove the universal vertices of~$H'$ that are from~$V(G)\setminus K_{s}$.
  \end{itemize}
  Once we have a clique path for the reduced graph, we can extend it to a clique path of~$H'$.
  For operation one, if the clique path of every component satisfies condition~\eqref{eq:2}, then we can concatenate them into a clique path of~$H'$ satisfying condition~\eqref{eq:2}.
  For operation two, once we have a clique path of the resulting graph that satisfies condition~\eqref{eq:2}, we can add these universal vertices to each clique.  Since these vertices are from~$V(G)\setminus K_{s}$, condition~\eqref{eq:2} remains satisfied.
  In the rest, $H'$ is connected since the first reduction is not applicable
  
  We take $\mathcal{M}$ to be the maximal vertex sets~$M$ such that (a)~$M\ne V(H')$; (b)~no vertex in~$H'[M]$ is universal; and (c)~$M$ is a module of~$H'$.   Note that $M$ is not a clique by condition (b).
  We argue that the modules in~$\mathcal{M}$ are pairwise disjoint and nonadjacent.
Suppose for contradiction that there are $M_{1}, M_{2} \in \mathcal{M}$ such that $M_{1}\cap M_{2} \ne \emptyset$.
By the definition of~$\mathcal{M}$ (maximality), neither $M_{1}$ nor~$M_{2}$ is a subset of the other.
By definition of modules, $M_{1}\cup M_{2}$ is a module of~$H'$ (any vertex in~$V(H')\setminus (M_{1}\cup M_{2})$ adjacent to~$(M_{1}\setminus M_{2})\cup (M_{2}\setminus M_{1})$ is adjacent to~$M_{1}\cap M_{2}$, and vice versa).
By condition (b), no vertex in~$H'[M_{1}\cup M_{2}]$ can be universal.
Thus, we must have $M_{1}\cup M_{2} = V(H')$:
otherwise neither $M_{1}$ nor~$M_{2}$ can be in~$\mathcal{M}$ because they are not maximal.
Then by the definition of modules, all the three sets~$M_{1}\setminus M_{2}$, $M_{2}\setminus M_{1}$, and~$M_{1}\cap M_{2}$ are modules of~$H'$.
Since $H'$ is connected, the three sets are complete to each other.
Since $H'$ is an interval graph, at least two of them are cliques, but then both $H'[M_{1}]$ and~$H'[M_{2}]$ have universal vertices, a contradiction.
If there is an edge between~$M_{1}$ and~$M_{2}$, then they are complete to each other because they are modules.  Since neither is a clique, there is a $C_4$, a contradiction.

Let $H''$ denote the graph obtained by replacing each module~$M$ in~$\mathcal{M}$ with a single vertex~$x_{M}$ from this module.  We choose $x_{M}$ from~$M\cap K_{o}$ if it is not empty.
We argue by contradiction that $H''$ is quasi-prime.  Suppose otherwise, then we can find a nontrivial module~$X$ of~$H''$ such that $H''[X]$ does not have universal vertices.  Then the vertex set
\[
  (X\cap V(H')) \cup \bigcup_{x_{M}\in X} M
\]
should be in~$\mathcal{M}$, contradicting the construction of~$H''$.
Since $H''$ is an induced subgraph of~$H'$, hence of~$H$, it is an interval graph and does not contain an annotated copy of any graph in Figure~\ref{fig:non-cag-to-interval}.
By Lemma~\ref{lem:quasi-prime}, $H''$ has a clique path~$\mathcal{K}'$ satisfying condition~\eqref{eq:2}.

For each module~$M\in \mathcal{M}$, since $M$ is not a clique, $N(M)$ is a clique.  Thus, the vertex~$x_{M}$ is simplicial in~$H''$.
If $M$ is disjoint from~$K_{o}$, we take an arbitrary clique path of~$H'[N[M]]$, and substitute it for the clique~$N_{H''}[x_{M}]$ in~$\mathcal{K}'$.
Since $M$ is disjoint from~$K_{o}$, this will not violate condition~\eqref{eq:2}.
  If $K_{o}$ is disjoint from all the modules in~$\mathcal{M}$, then we end with a clique path of~$H'$ and it satisfies condition~\eqref{eq:2}.
  
  Now suppose that there exists $M\in \mathcal{M}$ such that $M\cap K_{o}\ne \emptyset$.
  Since $K_{o}$ is a clique of~$H$ while modules in~$\mathcal{M}$ are pairwise disjoint and nonadjacent, $M$ is the only one intersecting $K_{o}$.
  If $N_{H'}(M)$ is disjoint from~$K_{s}$, then the clique path satisfies condition~\eqref{eq:2} as long as there exists a clique path of~$H'[M]$ that satisfies condition~\eqref{eq:2}.
  In this case, we may recursively consider $H'[M]$.

  In the sequel, neither $M\cap K_{o}$ nor~$N_{H'}(M)\cap K_{s}$ is empty.  
  Since $H'$ is an interval graph and~$M$ is not a clique, $N_{H'}(M)$ is a clique, and~$N_{H'}[x]\subseteq N_{H'}[y]$ for each pair of~$x\in M$ and~$y\in N_{H'}(M)$.
  Let $\langle K'_{1}, \ldots, K'_{s} \rangle$ be the clique path of~$H''$ and~$K'_{i} = N_{H''}[x_{M}]$.
  We note that if $1 < i < s$, then at least one of~$K'_{i-1}$ and~$K'_{i+1}$ is disjoint from~$K_{s}\cap N_{H'}(M)$.
  Recall that we have selected $x_{M}$ from~$K_{o}$.
  Since the clique path satisfies condition~\eqref{eq:2}, no vertex in~$K_{s}\cap N_{H'}(M)$ is in both $K'_{i-1}$ and~$K'_{i+1}$.
  If there are $v_{1}\in K_{s}\cap K'_{i-1}$ and~$v_{2}\in K_{s}\cap K'_{i+1}$, then $H$ contains an induced copy of Figure~\ref{fig:configuration-7}, where $x_{1}\in K'_{i-1}\setminus K'_{i}$, $x_{2}\in K'_{i+1}\setminus K'_{i}$, and~$x_{3}\in M\setminus N[x_{M}]$ (note that $x_{M}$ is not universal in~$H'[M]$).  
   
  For any two vertices~$v_{1}$ and~$v_{2}$ in~$K_{s}\cap N_{H'}(M)$, one of~$N_{H'}[v_{1}]$ and~$N_{H'}[v_{2}]$ must be the subset of the other.  Suppose that there are $x_{1} \in N_{H'}[v_{1}]\setminus N_{H'}[v_{2}]$ and~$x_{2} \in N_{H'}[v_{2}]\setminus N_{H'}[v_{1}]$.  Since $H'$ is an interval graph, $x_{1}$ and~$x_{2}$ are not adjacent.  Neither of them is in~$N_{H'}[M]$.  But then $\{x_{1}, v_{1}, v_{2}, x_{2}, u, y\}$, where $u$ is any vertex in~$M\cap K_{o}$ and~$y$ is any vertex in~$M\setminus N_{H'}[u]$ (note that $u$ is not universal in~$H'[M]$), makes an annotated copy of Figure~\ref{fig:configuration-7}.
  
  Case 1, there are two vertices~$u_{1}$ and~$u_{2}$ in~$M\cap K_{o}$ such that $N_{H'}[u_{1}]$ and~$N_{H'}[u_{2}]$ are incomparable.
  We argue that any clique path of~$H'$ satisfies the condition.
  Let $\langle K_{1}, \ldots, K_{\ell} \rangle$ be an clique path of~$N_{H'}[M]$.
  Assume without loss of generality that $\lk(u_{1}) < \lk(u_{2})$, hence $\rk(u_{1}) < \rk(u_{2})$.
  We take a vertex~$x_{1} \in K_{\lk(u_{1})}\setminus K_{\lk(u_{1})+1}$ and a vertex~$x_{2} \in K_{\rk(u_{2})}\setminus K_{\rk(u_{2})-1}$.  Note that $x_{1} u_{1} u_{2} x_{2}$ is an induced path.
  If $u_{1}\not\in K_{1}$, then $\{x_{1}, u_{1}, u_{2}, x_{2}, v, y\}$, where $v$ is any vertex in~$K_{s}\cap N_{H'}(M)$ and~$y$ is any vertex in~$K_{1}\setminus K_{2}$, induces an annotated copy of Figure~\ref{fig:configuration-3} or \ref{fig:configuration-8}, depending on whether $y$ is adjacent to~$x_{1}$.
  Thus, $u_{1}\in K_{1}$, and~$u_{2}\in K_{\ell}$ by a symmetric argument.
  If another vertex~$u\in M\cap K_{o}$ is in neither $K_{1}$ nor~$K_{\ell}$, then  $\{x_{1}, u_{1}, u_{2}, x_{2}, u, y\}$, where $y$ is any vertex in~$K_{s}\cap N_{H'}(M)$, induces an annotated copy of Figure~\ref{fig:configuration-6}.
  In summary, $M\cap K_{o}\subseteq K_{1}\cup K_{\ell}$.
On the other hand, since $H'$ does not contain an annotated copy of Figure~\ref{fig:configuration-8}, $N_{H'}[v] = N_{H'}[M]$ for all $v\in K_{s}\cap N_{H'}(M)$.  
This concludes this case.

Case 2, for any two vertices~$u_{1}$ and~$u_{2}$ in~$M\cap K_{o}$, one of~$N_{H'}[u_{1}]$ and~$N_{H'}[u_{2}]$ is a subset of the other.
Let $u$ be a vertex of~$M\cap K_{o}$ with the minimum degree.
Then $N_{H'}[u]\subseteq N_{H'}[u']$ for all $u'\in M\cap K_{o}$ by the assumption.
Since $H'$ does not contain any annotated copy of Figure~\ref{fig:configuration-1}, \ref{fig:configuration-3}, or \ref{fig:configuration-6}, $u$ cannot be a filled vertex in Figure~\ref{fig:non-end-interval} in the subgraph~$H'[M]$.   By Theorem~\ref{thm:end-interval}, there exists a clique path of~$H'[M]$ in which $u$ is in an end clique.  By the selection of~$u$, this end clique contains $M\cap K_{o}$.

Thus, in the clique path~$\mathcal{K}'$ of~$H'$, there is a maximal clique that contains $K_{s}\cap N_{H'}(M)$ while its predecessor or successor is disjoint from~$K_{s}\cap N_{H'}(M)$.
We can combine the clique path of~$H'[M]$ and~$\mathcal{K}'$ to produce a clique path of~$H'$ satisfying the condition.
This concludes the proof.
\end{proof}

By a \emph{forbidden configuration} we mean a minimal non-interval graph (with all vertices unspecified) or a graph in Figure~\ref{fig:non-cag-to-interval}.

\subsection{Forbidden induced subgraphs}

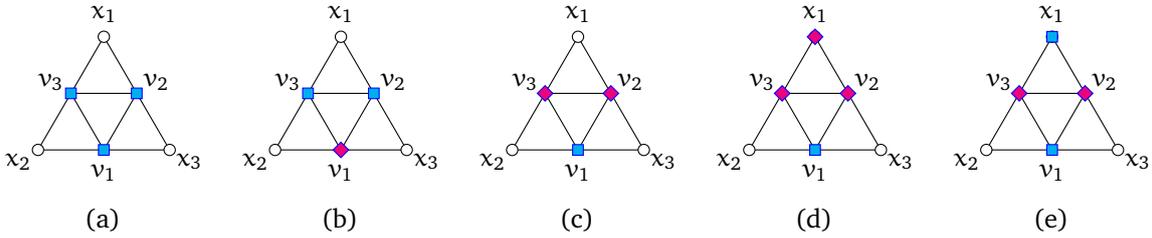
\begin{figure}[h]
  \centering \small
  \begin{subfigure}[b]{0.18\linewidth}
    \centering
    \begin{tikzpicture}[scale=.5]
      \foreach[count =\j] \i in {1, 2, 3} 
        \draw ({120*\i-90}:1) -- ({120*\i+30}:1) -- ({120*\i-30}:2) -- ({120*\i-90}:1);
      \foreach[count =\j] \i/\t in {1/, 2/, 3/} {
        \node[empty vertex] (u\i) at ({120*\i-30}:2) {};
        \node[a-vertex] (v\i) at ({120*\i-90}:1) {};

        \node at ({120*\i-30}:2.6) {$x_{\i}$};
        \node at ({120*\i+150}:1.6) {$v_{\i}$};
      }
    \end{tikzpicture}
    \caption{}
    \label{fig:split-non-cag-sun:a}
  \end{subfigure}
  \,
  \begin{subfigure}[b]{0.18\linewidth}
    \centering
    \begin{tikzpicture}[scale=.5]
      \foreach[count =\j] \i in {1, 2, 3} 
        \draw ({120*\i-90}:1) -- ({120*\i+30}:1) -- ({120*\i-30}:2) -- ({120*\i-90}:1);
      \foreach[count =\j] \i/\t in {1/a, 2/a, 3/b} {
        \node[empty vertex] (u\i) at ({120*\i-30}:2) {};
        \node[\t-vertex] (v\i) at ({120*\i-90}:1) {};

        \node at ({120*\i-30}:2.6) {$x_{\i}$};
        \node at ({120*\i+150}:1.6) {$v_{\i}$};
      }
    \end{tikzpicture}
    \caption{}
    \label{fig:split-non-cag-sun:b}
  \end{subfigure}
  \,
  \begin{subfigure}[b]{0.18\linewidth}
    \centering
    \begin{tikzpicture}[scale=.5]
      \foreach[count =\j] \i in {1, 2, 3} 
        \draw ({120*\i-90}:1) -- ({120*\i+30}:1) -- ({120*\i-30}:2) -- ({120*\i-90}:1);
      \foreach[count =\j] \i/\t in {1/b, 2/b, 3/a} {
        \node[empty vertex] (u\i) at ({120*\i-30}:2) {};
        \node[\t-vertex] (v\i) at ({120*\i-90}:1) {};

        \node at ({120*\i-30}:2.6) {$x_{\i}$};
        \node at ({120*\i+150}:1.6) {$v_{\i}$};
      }
    \end{tikzpicture}
    \caption{}
    \label{fig:split-non-cag-sun:c}
  \end{subfigure}
  \,
  \begin{subfigure}[b]{0.18\linewidth}
    \centering
    \begin{tikzpicture}[scale=.5]
      \foreach[count =\j] \i in {1, 2, 3} 
        \draw ({120*\i-90}:1) -- ({120*\i+30}:1) -- ({120*\i-30}:2) -- ({120*\i-90}:1);
      \foreach[count =\j] \i/\t in {1/b, 2/b, 3/a} {
        \node[empty vertex] (u\i) at ({120*\i-30}:2) {};
        \node[\t-vertex] (v\i) at ({120*\i-90}:1) {};

        \node at ({120*\i-30}:2.6) {$x_{\i}$};
        \node at ({120*\i+150}:1.6) {$v_{\i}$};
      }
      \node[b-vertex] at (u1) {};
    \end{tikzpicture}
    \caption{}
    \label{fig:split-non-cag-sun:d}    
  \end{subfigure}
  \,
  \begin{subfigure}[b]{0.18\linewidth}
    \centering
    \begin{tikzpicture}[scale=.5]
      \foreach[count =\j] \i in {1, 2, 3} 
        \draw ({120*\i-90}:1) -- ({120*\i+30}:1) -- ({120*\i-30}:2) -- ({120*\i-90}:1);
      \foreach[count =\j] \i/\t in {1/b, 2/b, 3/a} {
        \node[empty vertex] (u\i) at ({120*\i-30}:2) {};
        \node[\t-vertex] (v\i) at ({120*\i-90}:1) {};

        \node at ({120*\i-30}:2.6) {$x_{\i}$};
        \node at ({120*\i+150}:1.6) {$v_{\i}$};
      }
      \node[a-vertex] at (u1) {};
    \end{tikzpicture}
    \caption{}
    \label{fig:split-non-cag-sun:e}
  \end{subfigure}
  \caption{Some constitutions of vertices of a sun.
    The square, rhombus, and circle nodes are from~$K_{s}$, $K_{o}$, and~$S$, respectively.
  }
  \label{fig:split-non-cag-sun}
\end{figure}

We may derive minimal split graphs that are not circular-arc graphs in a similar way to Lemma~\ref{lem:alternative}.
However, the situation is far more complicated than Section~\ref{sec:hcag}.
The vertex set is now partitioned into three instead of two parts.
Each forbidden configuration of~$H$ corresponds to several graphs in region 3.
For example, there are more than ten partitions of the vertices of a sun, some of which are listed in Figure~\ref{fig:split-non-cag-sun}.
Note that the first three graphs in Figure~\ref{fig:split-non-cag-sun} are derived from graph~$(\overline{S_{3}})^{+}$, with different choices of the vertex~$s$.
The following focuses us on the vertex set of a forbidden configuration and simplicial vertices.

\begin{proposition}\label{lem:seed}
  Let $F$ be the vertex set of a forbidden configuration of~$H$.
  Then $G[F\cup S]$ is not a circular-arc graph.
\end{proposition}
\begin{proof}
  For each edge of~$H$ between~$F\cap K_{s}$ and~$F\cap K$, we take a witness.  Let $W$ denote the set of witnesses, and let $G_{0}$ denote the subgraph of~$G$ induced by~$F\cup W\cup \{s\}$.
  (One edge may be witnessed by more than one vertex in~$W$.)
  We argue that $F$ induces the same subgraph in~$G_{0}^{N[s]}$ as in~$H$.
  By construction, the subgraph in~$G_{0}^{N[s]}$, $H$, and~$G$ induced by~$F\setminus K_{s}$ are the same.
  For two vertices~$v\in F\cap K_{s}$ and~$x\in F\cap S$ are adjacent in~$G_{0}^{N[s]}$ and in~$H$ if and only if they are not adjacent in~$G$.
  For each pair of adjacent vertices~$u\in K_{s}\cap F$ and~$v\in K\cap F$ in~$H$, the selection of~$W$ ensures that they are adjacent in~$G_{0}^{N[s]}$.
  On the other hand, if two vertices~$u\in K_{s}\cap F$ and~$v\in K\cap F$ are not adjacent in~$H$, they cannot be adjacent in~$G_{0}^{N[s]}$.
  Now that $G_{0}^{N[s]}$ contains an annotated copy of  a forbidden configuration, $G_{0}$ is not a circular-arc graph by Lemma~\ref{lem:split-non-cag-but-interval}.
  Since the vertex set of~$G_{0}$ is $F\cup W\cup \{s\}$, the statement follows.
\end{proof}

To decrease the number of cases to consider, we need a few observations.
First, we can assume that every vertex in~$S$ is adjacent to a proper and nonempty subset of~$F\cap K$: we can reduce to a known case otherwise.


\begin{lemma}\label{lem:key}
  Let $F$ be the vertex set of a forbidden configuration of~$H$.
  If some vertex in~$S$ is adjacent to none or all the vertices in~$F\cap K$ in~$G$, then $G$ contains an induced copy of net$^\star$, a graph in region two, or~$(\overline{S_{k}})^{+}$ for some $k \ge 3$.
\end{lemma}
\begin{proof}
  By Proposition~\ref{lem:seed}, $G[F\cup S]$ is not a circular-arc graph by Lemma~\ref{lem:split-non-cag-but-interval}.
  By Theorem~\ref{thm:split-hcag}, $G[F\cup S]$ contains an induced copy of some graph in regions 1 and 2.
  We are done if $G[F\cup S]$ contains a graph in region 2. 
  In the rest, suppose that there exists $X\subseteq F\cup S$ such that $G[X]$ is isomorphic to~$\overline{S_{k}}$ for some $k$.
  Let $x\in S$ be a vertex that is adjacent to none or all the vertices in~$F\cap K$ in~$G$.
  Note that $x\not\in X$.
  If $X\cap K\subseteq F\cap K\subseteq N(x)$, then $G[X\cup \{x\}]$ is $(\overline{S_{k}})^{+}$ for some $k \ge 3$.
  Otherwise, $G[X\cup \{x\}]$ contains an induced net$^\star$ when $k \ne 4$, rising sun$^\star$ when $k = 4$, or sun$^\star$ when $k \ge 5$.
\end{proof}

In particular, Lemma~\ref{lem:key} covers the case when $H$ contains a minimal non-interval graph that is disjoint from~$K_{s}$ or~$K_{o}$; we can use $s$ as the special vertex required by Lemma~\ref{lem:key}.  Since $K_{o}$ is a clique and all vertices in~$S$ are simplicial in~$H$, the subgraph can be disjoint from~$K_{s}$ only when it is a net, sun, or rising sun.  On the other hand, every forbidden configuration in Figure~\ref{fig:non-cag-to-interval} intersects both $K_{s}$ and~$K_{o}$ by definition.

\begin{corollary}\label{cor:key} 
  If $H$ contains a minimal non-interval graph that is disjoint from~$K_{o}$ or~$K_{s}$, then $G$ contains an induced copy of net$^\star$, a graph in region two, or~$(\overline{S_{k}})^{+}$ for some $k \ge 3$.
\end{corollary}

Second, we extend Proposition~\ref{lem:h-maxclique} to the new setting.
Let $K'\subseteq K$ be a clique of~$H$.
A vertex~$w\in S$ is a \emph{witness} of~$K'$ if $w$ is adjacent to all the vertices in~$K'$ in~$H$; i.e.,
\[
  K_{o}\cap K'\subseteq N_{G}(w) \subseteq V(G)\setminus (K_{s}\cap K').
\]
The clique~$K'$ is then \emph{witnessed}.
Recall that every edge between~$K_{s}$ and~$K$ is witnessed, but edges among $K_{o}$ do not need witnesses.

\begin{proposition}\label{lem:h-maxclique-cag}
  If $H[K]$ has an unwitnessed clique~$K'$ with $|K'\cap K_{o}| \le 1$, then $H$ contains a sun of which 
  \begin{enumerate}[i)]
  \item all the degree-two vertices are from~$S$, and
  \item at most one vertex is from~$K_{o}$.
  \end{enumerate}
\end{proposition}
\begin{proof}
  Suppose that $K'$ is a smallest unwitnessed clique of~$H[K]$ such that $|K'\setminus K_{s}| \le 1$ and~$|K'|\ge 2$.
  By the construction of~$H$, every edge among~$K_{s}$ and every edge between~$K_{s}$ and~$K_{o}$ have a witness.
  Thus, $|K'| > 2$.
  We take three vertices~$v_{1}$, $v_{2}$, and~$v_{3}$ from~$K'$.
  By the selection of~$K'$, for each $i = 1, 2, 3$, the clique~$K'\setminus \{v_{i}\}$ has a witness~$x_{i}\in S$.
  Since $x_{i}$ is not a witness of~$K'$, it cannot be adjacent to~$v_{i}$ in~$H$.
  Then $\{v_{1}, v_{2}, v_{3}, x_{1}, x_{2}, x_{3}\}$ induces a sun.
\end{proof}

For suns, Lemma~\ref{lem:key} and Corollary~\ref{cor:key} 
  allow us to ignore the first three in Figure~\ref{fig:split-non-cag-sun}.  In (a), the vertex~$s$ is adjacent to all the three vertices from~$K$; in (b) and (c), the vertex~$x_{1}$ is adjacent to none and all three vertices from~$K$, respectively.
  Other possible configurations of the sun that are omitted from Figure~\ref{fig:split-non-cag-sun} can be reduced to the ones listed there
  by further observations.
We now derive the subgraphs of~$G$ corresponding to these configurations in Figure~\ref{fig:split-non-cag-sun}.  Similar to Lemma~\ref{lem:alternative}, we do not need witnesses for suns.

\begin{proposition}\label{lem:sun}
  If $H$ contains an annotated copy of any graph in Figure~\ref{fig:split-non-cag-sun}, then $G$ contains an induced $(\overline{S_{3}})^{+}$, sun$^\star$, or the graph in Figure~\ref{fig:the-weird}. 
\end{proposition}
\begin{proof}
  Let $F$ denote the vertex set of the sun.
  If $H[F]$ is an annotated copy of Figure~\ref{fig:split-non-cag-sun:a}, \ref{fig:split-non-cag-sun:c}, or \ref{fig:split-non-cag-sun:d}, then $G[F\cup \{s\}]$ is isomorphic to~$(\overline{S_{3}})^{+}$.
  For Figure~\ref{fig:split-non-cag-sun:a}, $G[F]$ is a net, of which all the three vertices are adjacent to~$s$.
  For Figures~\ref{fig:split-non-cag-sun:c} and~\ref{fig:split-non-cag-sun:d}, vertices~$x_{2}$, $x_{3}$, and~$s$ have degree one in~$G[F\cup \{s\}]$; their only neighbors are, respectively, $v_{3}$, $v_{2}$, and~$v_{1}$, which form a clique with $x_{1}$.
  If $H[F]$ is an annotated copy of Figure~\ref{fig:split-non-cag-sun:b}, then $G[F\cup \{s\}]$ is isomorphic to sun$^\star$.
The subgraph of~$G$ induced by~$F\cup \{s\}\setminus \{x_{1}\}$ is a sun, in which $x_{1}$ has no neighbor.


Henceforth, $H[F]$ is an annotated copy of Figure~\ref{fig:split-non-cag-sun:e}.
If there is a witness~$x$ of~$\{x_{1}, v_{2}, v_{3}\}$, then replacing $x_{1}$ with the witness leads to an annotated copy of Figure~\ref{fig:split-non-cag-sun:c}, which has been discussed above.
  Otherwise, by definition, we can find a witness~$w_{1}$ of the edge~$x_{1} v_{2}$; it is not adjacent to~$v_{3}$ as otherwise we are in the previous case.
  Likewise, we can find a witness~$w_{2}\in S\setminus N_{H}(v_{2})$ of the edge~$x_{1} v_{3}$.
  The subgraph of~$G$ induced by~$F\cup \{w_{1}, w_{2}\}$ is isomorphic to~${S_{4}}$, in which $s$ is adjacent to~$x_{1}$ and~$v_{1}$.  They together induce the graph in Figure~\ref{fig:the-weird}.
\end{proof}

Next we adapt Corollary~\ref{lem:simplicial-iff-i-helly}.  Note that in each forbidden configuration, a simplicial vertex has at most two neighbors.

\begin{lemma}\label{lem:simplicial-iff-i}
  If $G$ is not a circular-arc graph and~$H$ does not contain an annotated copy of any graph in Figure~\ref{fig:split-non-cag-sun},
  then $H$ contains a forbidden configuration~$F$ such that
  \begin{enumerate}[i)]
  \item all simplicial vertices of~$F$ are from~$S \cup K_{o}$, and
  \item if a simplicial vertex of~$F$ is from~$K_{o}$, then at least one of its neighbors in~$F$ is from~$K_{o}$.
  \end{enumerate}
\end{lemma}
\begin{proof}
  By Lemma~\ref{lem:split-non-cag-but-interval}, $H$ contains an annotated copy of a forbidden configuration; let it be $F$.
  We are done if all simplicial vertices of~$F$ are from~$S$.
  Suppose that a simplicial vertex~$x$ of~$F$ is from~$K$.
  Note that $N_{F}(x)$ contains one or two vertices and they are not simplicial.
  If there exists a witness~$y$ of~$N_{F}[x]$, then we can replace $x$ with~$y$.
  Henceforth we assume $N_{F}[x]$ is not witnessed.
  As a result, if $x\in K_{s}$, then
  $N_{F}(x)$ must consist of two vertices, and they are both from~$K_{o}$ by Proposition~\ref{lem:h-maxclique-cag}.
  Let $N_{F}(x) = \{v_{1}, v_{2}\}$.
  For $i = 1, 2$, we can find a witness~$w_{i}$ of the edge~$x v_{i}$.
  By checking all the forbidden configurations we note that $v_{1}$ and~$v_{2}$ always has another common neighbor in~$F$, denoted as~$y$.
  Then $H[\{x, y, v_{1}, v_{2}, w_{1}, w_{2}\}]$ is an annotated copy of Figure~\ref{fig:split-non-cag-sun:c}, \ref{fig:split-non-cag-sun:d}, or \ref{fig:split-non-cag-sun:e}, a contradiction.
  Thus, $x\in K_{o}$.  Since $N_{F}[x]$ is not witnessed,
at least one vertex in~$N_{F}(x)$ is from~$K_{o}$ by Proposition~\ref{lem:h-maxclique-cag}.
\end{proof}

With these observations, we are able to correlate forbidden configurations with minimal split graphs that are not circular-arc graphs.
We first deal with the forbidden configurations in Figure~\ref{fig:non-cag-to-interval}.
Since no forbidden configuration has a clique of order more than four, the number of vertices from the clique~$K_{o}$ is at most four.
We try to reduce a forbidden configuration in~$G^{N[s]}$ to~$G^{K}$, or~$G^{N[s']}$ for another simplicial vertex~$s'$ that has more neighbors in this forbidden configuration.
This is more convenient than translating the forbidden configuration to a forbidden induced subgraph in~$G$. 
Note that $K_{s}\cup S\setminus \{s\}$ induces the same subgraph in~$H$ and~$G^{K}$.

\begin{lemma}\label{lem:interval-configurations}
  If $H$ contains an annotated copy of any graph in Figure~\ref{fig:non-cag-to-interval},
  then $G$ contains an induced copy of net$^\star$, the graph in Figure~\ref{fig:long-claw-derived}, Figure~\ref{fig:whipping-top-derived}, or Figure~\ref{fig:the-weird}, $S^{1}_{k}$, $S^{2}_{k}$, or~$\overline{S_{k+1}^{+}}, k \ge 2$.
\end{lemma}
\begin{proof}
  If $H$ contains an annotated copy of a sun as in Figure~\ref{fig:split-non-cag-sun}, then $G$ contains an induced copy of sun$^\star$,  $(\overline{S_{3}})^{+}$, or the graph in Figure~\ref{fig:the-weird} by Proposition~\ref{lem:sun}.  Hence, we may assume otherwise.
  In each graph of Figure~\ref{fig:non-cag-to-interval}, no unspecified simplicial vertex is adjacent to all the vertices in~$K_{o}$.
  Since $H$ does not contain any induced sun, all the unspecified simplicial vertices are from~$S$ by Lemma~\ref{lem:simplicial-iff-i}.
  Let $F$ be the vertex set of the forbidden configuration.
  If $H[F]$ is an annotated copy of Figure~\ref{fig:configuration-7} or \ref{fig:configuration-8}, then $N_{G}(x_{3})$ is disjoint from~$K\cap F$.
  The statement follows from Lemma~\ref{lem:key}.
  For most of the others, we reduce to a minimal non-interval subgraph of~$G^{K}$ and use Lemma~\ref{lem:alternative}.
  If the subgraph is a sun or hole, we find another vertex to form a graph in Figure~\ref{fig:split-non-cag}.
  When we process a forbidden configuration, we assume that $H$ does not contain an annotated copy of a forbidden configuration discussed previously.
  
  Figure~\ref{fig:configuration-1}.
  By definition, there are witnesses~$w_{1}$, $w_{2}$, and~$w_{3}$ for the edges~$v u$, $v x_{2}$, and~$v x_{3}$, respectively.
  In~$G^{K}$, the vertex~$u$ is adjacent to~$x_{1}$, $x_{4}$, $w_{2}$, and~$w_{3}$, and then to~$v$, $x_{2}$, and~$x_{3}$.
  On the other hand, $u$ is not adjacent to~$w_{1}$ in~$G^{K}$.
  Thus, $G^{K}[F\cup \{w_{1}\}]$ is isomorphic to a whipping top, while $w_{2}$ and~$w_{3}$ are witnesses of~$\{v, u, x_{2}\}$ and~$\{v, u, x_{3}\}$, respectively.
  By Lemma~\ref{lem:alternative}, $G$ contains an induced copy of Figure~\ref{fig:whipping-top-derived}.
  
   Figure~\ref{fig:configuration-2}.
   By assumption, there are a witness~$w_{1}$ of~$\{v, x_{2}, x_{3}\}$ and a witness~$w_{2}$ of~$\{v, u\}$.
   Since $\{x_{1}, x_{3}, x_{4}, v, u, w_{1}\}$ does not induce the configuration in Figure~\ref{fig:configuration-7}, $x_{3}\in K_{o}$.
   Since $\{x_{3}, x_{4},v, u, w_{1}, w_{2}\}$ does not induce the configuration in Figure~\ref{fig:configuration-8}, $w_{2}$ and~$x_{3}$ must be adjacent in~$H$.
   In~$G^{K}$, the vertex~$u$ is adjacent to~$x_{1}$, $x_{4}$, $w_{1}$, and then to~$x_{2}$ and~$v$; the vertex~$x_{3}$ is adjacent to~$x_{4}$, and then to~$u$ and~$v$.
   Both~$u$ and~$x_{3}$ are adjacent to~$s$ in~$G^{K}$.
  Thus, $G^{K}[\{w_{2}, x_{1}, x_{2}, v, x_{3}, s, u\}]$ is isomorphic to a whipping top, while $w_{1}$ and~$x_{4}$ are witnesses of~$\{v, u, x_{2}\}$ and~$\{v, u, x_{3}\}$, respectively.
  By Lemma~\ref{lem:alternative}, $G$ contains an induced copy of Figure~\ref{fig:long-claw-derived}.

  Figure~\ref{fig:configuration-3}.
  If $x_{2}\in K_{o}$, then $x_{3}\in K_{s}$ and~$N_{G}(x_{4})$ is disjoint from~$K\cap F$.  The statement follows from Lemma~\ref{lem:key}.
  Thus, $x_{2}\in K_{s}$; for the same reason, $x_{3}\in K_{s}$.
  By assumption, there are a witness~$w_{1}$ of~$\{v, u, x_{2}\}$ and a witness~$w_{2}$ of~$\{v, u, x_{3}\}$.
  In~$G^{K}$, the vertex~$u$ is adjacent to~$x_{1}$,~$x_{4}$, and then to~$x_{2}$, $x_{3}$, and~$v$.
  Thus, $G^{K}[\{s, w_{1}, x_{2}, u, x_{3}, w_{2}, v\}]$ is isomorphic to a whipping top, while $x_{1}$ and~$x_{4}$ are witnesses of~$\{v, u, x_{2}\}$ and~$\{v, u, x_{3}\}$, respectively.
  By Lemma~\ref{lem:alternative}, $G$ contains an induced copy of Figure~\ref{fig:long-claw-derived}.
  
  Figure~\ref{fig:configuration-4}.
  Consider first the case that $F$ consists of six vertices and two of them are from~$K_{o}$.
  By assumption, there is a witness~$w$ of~$\{v, x_{2}, x_{3}\}$.
  The vertex set~$\{w, x_{3}, x_{0}, v, x_{1}, x_{2}\}$ induces a net in~$G$, and none of the degree-one vertices is adjacent to~$u$.
  Thus, we find an induced $\overline{S_{3}^{+}}$ in~$G$.
  Henceforth, we may assume that there exists only one vertex from~$K_{o}$; otherwise, we may drop the vertex~$u$ and consider the rest.
  Let $p = |F| - 3$.  
  By assumption, there exists a witness~$w_{1}$ of~$\{v, u, x_{p}\}$, and for each $i = 2, \ldots, p$, there exists a witness~$w_{i}$ of~$\{v, x_{i}, x_{i+1}\}$.
  In~$G^{K}$, the vertex~$u$ is adjacent to~$x_{0}$, $x_{1}$, and~$w_{i}, i\ge 2$, and hence to~$v$ and~$x_{i}, i\ge 2$.  Thus, $G^K[F\cup \{w_{1}\}]$ is isomorphic to the $\ddag$ graph on at least seven vertices, and for~$i\ge 2$, the vertex~$w_{i}$ is the witness of~$\{v, u, x_{i}, x_{i+1}\}$.
  By Lemma~\ref{lem:alternative}, $G$ contains an induced copy of~$S^{2}_{p}$.
  

  Figure~\ref{fig:configuration-6}.  Suppose first that $|F| = 6$.  Note that $x_{1}$ and~$x_{3}$ are symmetric in this case.  We consider the number of vertices in~$F$ from~$K_{o}$.
  \begin{itemize}
  \item If $|K_{o}\cap F| = 1$, then there is a witness~$w$ of~$\{v, u, x_{2}, x_{4}\}$.  In~$G^{K}$, the vertex~$u$ is adjacent to all the other vertices in~$F$ but not $w$.  Thus, $G^{K}[u, w, x_{1}, x_{2}, x_{3}, x_{4}]$ is isomorphic to a sun. 
    By Lemma~\ref{lem:alternative}, $G[u, w, x_{1}, x_{2}, x_{3}, x_{4}]$ is isomorphic to~$\overline{S_{3}}$.
    Since $v$ is adjacent to none of~$x_{1}$, $x_{2}$, or~$w$ in~$G$, the subgraph~$G[F\cup \{w\}]$ is isomorphic to~$\overline{S_{3}^{+}}$.
  \item If $|K_{o}\cap F| = 2$, then $N_{G}(x_{1})$ or~$N_{G}(x_{3})$ is disjoint from~$K\cap F$, and it follows from Corollary~\ref{cor:key}.
  \item If $|K_{o}\cap F| = 3$, then $v$ is adjacent to neither $x_{1}$ nor~$x_{2}$ in~$G$.  Thus, $G[F\cup\{s\}]$ is isomorphic to~$\overline{S_{3}^{+}}$. 
  \end{itemize}
  In the rest, $|F| \ge 7$, and let $p = |F| - 2$.
  The subgraph~$H[F\setminus \{v\}]$ is an annotated copy of Figure~\ref{fig:configuration-4} if $x_{2}\in K_{s}$.
  We may assume that $x_{p}\in K_{s}$; otherwise, the subgraph~$H[F\setminus \{u\}]$ is also an annotated copy of Figure~\ref{fig:configuration-6}.
  By assumption,
  for each $i = 4, \ldots, p-1$, there exists a witness~$w_{i}$ of~$\{v, x_{2}, x_{i}, x_{i+1}\}$, and there exists a witness~$w_{p}$ of~$\{v, u, x_{p}\}$.
  Since $H[\{v, u, w_{p}, x_{1}, x_{2}, x_{3}\}]$ does not induce the configuration of Figure~\ref{fig:configuration-8}, the vertex~$w_{p}$ must be adjacent to~$x_{2}$.
  In~$G^{K}$, the vertex~$u$ is adjacent to~$x_{1}$, $x_{3}$, and~$w_{i}, i\ge 4$, and hence to~$v$ and~$x_{i}, i\ge 2$; the vertex~$x_{2}$ is adjacent to only $x_{3}$, and then $x_{4}$.
  Thus, $F\setminus \{x_{3}\}\cup \{s, w_{p}\}$ induces a $\ddag$ graph on at least eight vertices, and~$x_{3}$ is a witness of~$\{v, u, x_{2}, x_{4}\}$, while $w_{i}, i\ge 4$, is a witnesses of~$\{v, u, x_{i}, x_{i+1}\}$.
  By Lemma~\ref{lem:alternative}, $G$ contains an induced copy of~$S^{2}_{p+2}$.
\end{proof}

We now deal with minimal non-interval subgraphs of~$H$, for which we use the labeling of vertices given by Figure~\ref{fig:non-interval}. 


\begin{theorem}\label{thm:h-not-interval}
  Let $G$ be a split graph.  If $G$ is not a circular-arc graph, then $G$ contains an induced copy of net$^\star$, the graph in Figure~\ref{fig:long-claw-derived}, Figure~\ref{fig:whipping-top-derived}, or Figure~\ref{fig:the-weird}, $S^{1}_{k}$, $S^{2}_{k}$, or~$\overline{S_{k+1}^{+}}, k \ge 2$.
\end{theorem}
\begin{proof}
  By Lemma~\ref{lem:split-non-cag-but-interval}, $H$ contains an annotated copy of a forbidden configuration.
  If $H$ contains an annotated copy of any graph in Figure~\ref{fig:non-cag-to-interval} or in Figure~\ref{fig:split-non-cag-sun}, then the statement follows from Lemma~\ref{lem:interval-configurations} and Proposition~\ref{lem:sun}, respectively.
  Hence, we assume that $H$ is not an interval graph, and it contains a forbidden configuration with the two properties specified in Lemma~\ref{lem:simplicial-iff-i}.
  Let $F$ be the vertex set of this forbidden configuration.
  We may also assume without loss of generality that $F$ intersects $K_{o}$; otherwise we can use Lemma~\ref{lem:key}.
  Since $H$ does not contain any annotated copy of a graph in Figure~\ref{fig:split-non-cag-sun}, every clique with at most one vertex from~$K_{o}$ is witnessed by Proposition~\ref{lem:h-maxclique-cag}.
  

  First, we consider the case that $|F\cap K_{o}| = 1$.
  Since $F$ satisfies Lemma~\ref{lem:simplicial-iff-i}, all its simplicial vertices are from~$S$.

  \begin{itemize}
  \item
    Hole.  Let it be $v_{1} v_{2} \cdots v_{\ell}$, and assume without loss of generality that $v_{1}\in K_{o}$.  By definition, for $i = 1, \ldots, \ell$, there are witnesses~$w_{i}$ of the edge~$v_{i}v_{i+1\pmod \ell}$.  In~$G^{K}$, the vertex~$v_{1}$ is adjacent to~$w_{2}$, $\ldots$, $w_{\ell - 1}$, and hence $v_{2}$, $\ldots$, $v_{\ell}$.  Thus, $G^{K}[\{s, v_{1}, w_{1}, v_{2}, \ldots, v_{\ell}, w_{\ell}\}]$ is isomorphic to the $\dag$ graph of order~$\ell + 3 \ge 7$.
    By Lemma~\ref{lem:alternative}, $G$ contains an induced copy of~$S^{1}_{\ell-1}$.
  \item Long claw.
    Since $H$ does not contain the configuration in Figure~\ref{fig:configuration-1}, $F\cap K_{o}\not\subseteq \{v_{1}, v_{2}, v_{3}\}$.
    Thus, $F\cap K_{o} = \{v_{0}\}$.
    For $i = 1, 2, 3$, there is a witness~$w_{i}$ of~$v_{0}v_{i}$.  They are distinct because $v_{1}$, $v_{2}$, and~$v_{3}$ are pairwise nonadjacent.
    Then $G^K[\{v_{0}, \ldots, v_{3}, w_{1}, w_{2}, w_{3}\}]$ is isomorphic to a long claw, and the vertex~$x_{i}, i = 1, 2, 3$, is a witness of the edge~$v_{0} v_{i}$.
    (The roles of~$x_{i}$ and~$w_{i}$ are switched.)
    By Lemma~\ref{lem:alternative}, $G$ contains an induced copy of Figure~\ref{fig:long-claw-derived}.
  \item Whipping top.
    Since $H$ does not contain the configuration in Figure~\ref{fig:configuration-2} or \ref{fig:configuration-3}, $F\cap K_{o}\not\subseteq \{v_{1}, v_{2}, v_{3}\}$.
    Thus, $F\cap K_{o} = \{v_{0}\}$.    
    By assumption, there are a witness~$w_{1}$ of~$\{v_{0}, v_{1}, v_{2}\}$ and a witness~$w_{2}$ of~$\{v_{0}, v_{2}, v_{3}\}$.
    In~$G^K$, the vertex~$v_{0}$ is adjacent to~$s$ and its neighbors in~$F$ are $x_{2}$ and~$v_{2}$. 
    Thus, $G^K[\{s, v_{0}, \ldots, v_{3}, x_{1}, x_{3}\}]$ is a long claw, where the edges~$v_{2} v_{0}$, $v_{2} v_{1}$, and~$v_{2} v_{3}$ are witnessed by~$x_{2}$, $w_{1}$, and~$w_{3}$, respectively.
    By Lemma~\ref{lem:alternative}, $G$ contains an induced copy of Figure~\ref{fig:long-claw-derived}.
  \item $\dag$ graph.
    If $u_{0}$ is the only vertex in~$F\cap K_{o}$, then $x_{2}$ is adjacent to all the vertices in~$F\cap K$, and the statement follows from Lemma~\ref{lem:key}.
    It is similar if the only vertex in~$F\cap K_{o}$ is $x_{1}$ or~$x_{3}$.
Thus, $p \ge 3$, and~$F\cap K_{o}$ is $v_{i}$ with $1 < i < p$.
Since $H$ does not contain an annotated copy of Figure~\ref{fig:configuration-4}, $v_{i}$ must be adjacent to both $v_{1}$ and~$v_{p}$; i.e., $p = 3$ and~$i = 2$.
Then $G[F\cup \{w_{1}, w_{2}\}]$, where $w_{1}$ and~$w_{2}$ are witnesses of~$\{u_{0}, v_{1}, v_{2}\}$ and~$\{u_{0}, v_{2}, v_{3}\}$, respectively, is isomorphic to Figure~\ref{fig:the-weird}.  
  \item $\ddag$ graph.
    Since $H$ does not contain an annotated copy of Figure~\ref{fig:configuration-7}, the vertex in~$F\cap K_{o}$ must be one of~$\{v_{1}, v_{p}, u_{1}, u_{2}\}$.
    Consider first $v_{1}$, and it is symmetric for~$v_{p}$.
    Since $H$ does not contain an annotated copy of Figure~\ref{fig:configuration-4} or \ref{fig:configuration-6}, $p = 1$.  
    In $G^{K}$, the vertex~$v_{1}$ is adjacent to~$x_{2}, u_{1}, u_{2}$ in~$F$, and~$s$.
    Thus, $G^{K}[\{s, v_{1}, x_{1}, u_{1}, x_{3}, u_{2}\}]$ is isomorphic to a net, and~$x_{2}$ is the witness of~$\{v_{1}, u_{1}, u_{2}\}$.
    By Lemma~\ref{lem:alternative}, $G$ contains an induced copy of ~$S^{1}_{1}$.
    Consider then $u_{1}$, and it is symmetric for~$u_{2}$.
    It is the same as above when $p = 1$, and hence we assume $p \ge 2$.
    By assumption, for $i = 1, \ldots, p-1$, there is a witness~$w_{i}$ of~$\{u_{1}, u_{2}, v_{i}, v_{i+1}\}$.
    In~$G^K$, the vertex~$u_{1}$ is adjacent to~$x_{3}$,$u_{2}$, $v_{p}$, and~$s$.
    Thus, $G^K[\{x_{2}, u_{2}, x_{1}, v_{1}, \ldots, v_{p}, u_{1}, s\}]$ is isomorphic to a \dag.
    The vertex~$x_{3}$ is a witness of~$\{u_{1}, u_{2}, v_{p}\}$, while $w_{i}, i = 1, \ldots, p-1$ is a witness of~$\{u_{2}, v_{i}, v_{i+1}\}$.
    By Lemma~\ref{lem:alternative}, $G$ contains an induced copy of~$S^{2}_{p-1}$.
  \end{itemize}
  
  From now on, $|F\cap K_{o}| \ge 2$.
  By Proposition~\ref{lem:seed}, $G[F\cup S]$ is not a circular-arc graph.
  If there exists a vertex~$x\in S$ such that $|N_G(x)\cap F| > |K_{s}|$, we can use the induction hypothesis.
  Note that $G[F\cup S]^{(N(x)\cap F)\cup \{x\}}$ contains an annotated copy of a forbidden configuration (Lemma~\ref{lem:split-non-cag-but-interval}), and~$|F\setminus N_G(x)| < |F\cap K_{o}|$.

  \begin{itemize}
  \item Long claw.
    Since $K_{o}$ is a clique, one of the degree-two vertices must be in~$K_{o}$; assume it is $v_{1}$.  
    Since $H$ does not contain the configuration Figure~\ref{fig:configuration-1}, $F\cap K_{o} = \{v_{0}, v_{1}\}$.  Then $\{v_{1}, v_{2}, v_{3}\}\subseteq N_G(x_{1})$ and~$|N_G(x_{1})\cap F| > |K_{s}|$.
    We can use the induction hypothesis on the graph~$G^{N[x_1]}$.
  \item Whipping top.
    Since $H$ does not contain the configuration Figure~\ref{fig:configuration-3}, if $v_{2}\in K_{o}$, then $v_{0}\in K_{o}$ and~$x_{2}\not\in K_{o}$.  We can use the induction hypothesis on the graph~$G^{N[x_2]}$.
    In the sequel, $v_{2}\not\in K_{o}$.
    Since $H$ does not contain the configuration Figure~\ref{fig:configuration-2}, $v_{1}\not\in K_{o}$.
    Thus, $F\cap K_{o}$ comprises $v_{0}$ and one of~$x_{1}$ and~$x_{3}$.  We consider $\{x_{1}, v_{0}\}$ and the other is symmetric.
    By definition, there is a witness~$w$ of~$x_{1}v_{1}$.
    If $w v_{0}\in E(H)$, then we can replace $x_{1}$ with $w$ to get an induced whipping top in~$H$ that has only one vertex from~$K_{o}$.
    Otherwise, $\{w, v_{1}, x_{2}, v_{2}, x_{3}, v_{0}\}$ induces a net in~$H$, and only one vertex in them is from~$K_{o}$.  In either case, we can use the induction hypothesis.
  \item $\dag$ graph.
    Since $H$ does not contain the configuration Figure~\ref{fig:configuration-8}, at least one of~$u_{0}$, $v_{1}$, and~$v_{p}$ is in~$K_{o}$.
    As discussed above, if $K_{o}$ contains $u_{0}$ but not $x_{2}$, we can use the induction hypothesis on the graph~$G^{N[x_2]}$.
    It is similar if $K_{o}$ contains $v_{1}$ but not $x_{1}$
    or contains $v_{p}$ but not $x_{3}$.
    Thus, $F\cap K_{o}$ is either $\{u_{0}, x_{2}\}$, $\{v_{1}, x_{1}\}$, or~$\{v_{p}, x_{3}\}$.
    The three pairs are symmetric when $p = 2$.
    On the other hand, if $p > 2$ and~$F\cap K_{o} = \{v_{1}, x_{1}\}$ or~$\{v_{p}, x_{3}\}$, then $H$ contains an annotated copy of Figure~\ref{fig:configuration-4}.
    Thus, it suffices to consider $F\cap K_{o} = \{u_{0}, x_{2}\}$.
    By assumption, for $i = 1, \ldots, p-1$, there is a witness~$w_{i}$ of the clique~$\{u_{0}, v_{i}, v_{i+1}\}$.
    The subgraph of~$G$ induced by~$\{u_{0}, v_{1}, \ldots, v_{p}, x_{1}, w_{1}, \ldots, w_{p-1}, x_{3}\}$ is isomorphic to a $\overline{S_{p+1}}$.
    Since $x_{2}$ is adjacent to ~$\{u_{0}, v_{1}, \ldots, v_{p}\}$ and not to~$\{x_{1}, w_{1}, \ldots, w_{p-1}, x_{3}\}$, the graph~$G$ contains an induced $\overline{S_{p+1}^{+}}$.
  \item Sun.
    We can use Proposition~\ref{lem:sun} if $F\cap K_{o} = \{u_{1}, u_{2}\}$ or~$\{x_{2}, u_{1}, u_{2}\}$, and Proposition~\ref{lem:key} if $F\cap K_{o} = \{v_{1}, u_{1}, u_{2}\}$.
    The only remaining case is $F\cap K_{o}$ comprises a degree-two vertex and one of its neighbors.
    Without loss of generality, assume $F\cap K_{o} = \{u_{1}, x_{2}\}$.
    Let $w$ be a witness of~$x_{2}$ and~$u_{2}$.
    If $w$ is adjacent to~$u_{1}$ in~$H$, then we can replace $x_{2}$ with $w$ and use the induction hypothesis.
    In the rest, $w u_{1}\not\in E(H)$.
    In~$G^{K}$, the vertex~$u_{1}$ is adjacent to~$w$ and~$x_{3}$, and hence to~$u_{2}$ and~$v_{1}$, and~$x_{2}$ is adjacent to~$x_{1}$ and~$x_{3}$, and hence to~$u_{2}$ and~$v_{1}$.
    Thus, $G^{K}[\{s, u_{1}, x_{2}, w, u_{2}, v_{1}, x_{1}\}]$ is isomorphic to a rising sun, and~$x_{3}$ is a witness of$G^{K}[\{v_{1}, u_{1}, u_{2}, x_{2}\}]$.
    By Lemma~\ref{lem:alternative}, $G$ contains an induced copy of ~$S^{2}_{1}$.    
  \item $\ddag$ graph with $p \ge 2$.
    First consider $\{u_{1}, u_{2}\}\subseteq F\cap K_{o}$.
    If $x_{2}\not\in F\cap K_{o}$, then we can use the induction hypothesis on the graph~$G^{N[x_{2}]}$.
    Hence, $F\cap K_{o} = \{x_{2}, u_{1}, u_{2}\}$.
    If for some $i\in \{1, \ldots, p - 1\}$, the clique ~$\{u_{1}, u_{2}, v_{i}, v_{i+1}\}$ is not witnessed, then
    we can find a witness~$w'$ of~$\{u_{1}, v_{i}, v_{i+1}\}$ and a witness~$w''$ of~$\{u_{2}, v_{i}, v_{i+1}\}$; note that $u_{2} w', u_{1} w''\not\in E(H)$.
    Then $H[\{x_{2}, u_{1}, u_{2}, w', w'', v_{i}\}]$ is an annotated copy of \ref{fig:split-non-cag-sun}d.
    In the rest, for each $i = 1, \ldots, p - 1$, there exists a witness~$w_{i}$ of~$\{u_{1}, u_{2}, v_{i}, v_{i+1}\}$.
    Thus, $G[F\cup \{w_{1}, w_{2}, \ldots, w_{p-1}, s\}]$ is isomorphic to~$\overline{S_{p+2}^{+}}$.

    Henceforth, $F\cap K_{o}$ contains at most one of~$u_{1}$ and~$u_{2}$.  We may assume without loss of generality that $u_{2}\not\in F\cap K_{o}$.

    If $F\cap K_{o} = \{u_{1}, x_{2}\}$, let $w$ be a witness of~$\{x_{2}, u_{2}\}$.
    Since $H[\{x_{3}, u_{2}, v_{1}, u_{1}, x_{2}, w\}]$ is not an annotated copy of Figure~\ref{fig:configuration-8}, $w u_{1}\in E(H)$.
    Then we can replace $x_{2}$ with $w$ and use the induction hypothesis.
    If $F\cap K_{o} = \{u_{1}, x_{1}\}$, we take a witness~$w$ of~$\{x_{1}, v_{1}\}$.
    If $w u_{1}\in E(H)$, then we can replace $x_{1}$ with $w$ to get an induced $\ddag$.
    Otherwise, 
    $H[\{x_{2}, u_{1}, w, v_{1}, v_{2}, \ldots, v_{p}, x_{3}\}$ is isomorphic to~$\dag$.  In either case, this new non-interval subgraph contains only one vertex from~$K_{o}$, and we can use the induction hypothesis.

    The only remaining  case is $F\cap K_{o} = \{u_{1}, v_{i}\}$ for some $i = 1, \ldots, p$.
    If $p > 2$ (i.e., $|F| > 7$), then $H$ contains an annotated copy of Figure~\ref{fig:configuration-4} when $i = p$, or an annotated copy of Figure~\ref{fig:configuration-6} when $i < p$.
    Thus, $|F| = 7$.
    Again, $H$ contains an annotated copy of Figure~\ref{fig:configuration-6} when $F\cap K_{o} = \{u_{1}, v_{1}\}$.
    Hence, $F\cap K_{o} = \{u_{1}, v_{2}\}$.
    By assumption, there is a witness~$w$ of~$\{u_{1}, u_{2}, v_{1}\}$.
    Since $H[\{x_{1}, v_{1}, u_{2}, x_{2}, v_{2}, w\}]$ is not an annotated copy of Figure~\ref{fig:configuration-7}, we have $w v_{2}\in E(H)$.
    Then $G[F\cup \{w, s\}]$ is isomorphic to the graph in Figure~\ref{fig:the-weird}.
  \end{itemize}

  The proof is complete.  
\end{proof}

\appendix
\section{Sketch of the certifying recognition algorithms}
We obtain a certifying recognition algorithm for recognizing circular-arc graphs by making the proofs in Section~\ref{sec:cag} constructive.  
We outline the algorithm in Figure~\ref{alg:certifying-recognition}.

\begin{figure}[h!]
  \centering
  \begin{tikzpicture}
    \path (0,0) node[text width=.85\textwidth, inner xsep=20pt, inner ysep=10pt] (a) {
      \begin{minipage}[t!]{\textwidth}
        \begin{tabbing}
          AAA\=Aaa\=aaa\=Aaa\=MMMMMAAAAAAAAAAAA\=A \kill
          1. \> find a simplicial vertex $s$ and construct the graph~$G^{N[s]}$;
          \\
          2. \> \textbf{if} $G^{N[s]}$ is an interval graph \textbf{then}
          \\
          2.1. \>\> try to find an interval model $\mathcal{I}$ of $G^{N[s]}$ which satisfies condition~\eqref{eq:1};
          \\
          2.2. \>\> \textbf{if} step~2.1 succeeds \textbf{then}
          \\
          2.2.1. \>\>\> translate $\mathcal{I}$ into a circular-arc model~$\mathcal{A}$ of~$G$;
          \\
          2.2.2. \>\>\> \textbf{return} $\mathcal{A}$;
          \\
          2.3. \>\> \textbf{else} find a forbidden configuration~$X$ in $G^{N[s]}$;
          \\
          3. \> \textbf{else} find a minimal non-interval graph~$X$ of~$G^{N[s]}$;
          \\
          4. \> translate $X$ into a minimal forbidden induced subgraph~$F$ of~$G$;
          \\
          5. \> \textbf{return} $F$.
        \end{tabbing}
      \end{minipage}
    };
    \draw[draw=gray!60] (a.north west) -- (a.north east) (a.south west) -- (a.south east);
  \end{tikzpicture}
  \caption{Outline of the certifying recognition algorithm.}
  \label{alg:certifying-recognition}
\end{figure}

Clearly, we are able to construct $G^{N[s]}$ in polynomial time.
If $G^{N[s]}$ is an interval graph, we try to find an interval model satisfying condition~\eqref{eq:1}.
This can be done as the proof of Lemma~\ref{lem:split-non-cag-but-interval} can be easily turned into a polynomial algorithm.
When there is an annotated copy of a forbidden configuration in Figure~\ref{fig:non-cag-to-interval}, it is explicitly given in the proof.
If we succeed, $G$ is a circular-arc graph. 
We can use Theorem~\ref{thm:correlation} to construct a circular-arc model~$\mathcal{A}$ for~$G$.
Step~3 calls the algorithm of Lindzey and McConnell~\cite{lindzey-16-find-forbidden-subgraphs} to find a minimal non-interval subgraph when $G^{N[s]}$ is not an interval graph.
Step~4 translates the forbidden configuration~$X$ to a minimal forbidden induced subgraph of~$G$, which also can be done in polynomial time.


The only obstacle toward a linear-time implementation is to decide whether two vertices double overlap, a crucial step in the construction of~$G^{N[s]}$.
The procedure of McConnell~\cite[Theorem~7.11]{mcconnell-03-recognition-cag} for this purpose works only when the input graph is a circular-arc graph.  For a certifying algorithm, we need the answer even the graph is not.
As long as we take~$s$ to be a vertex with the minimum degree, the size of~$G^{N[s]}$ is upper bounded by that of~$G$.  On the other hand, this is not the case for the recognition of Helly circular-arc graphs: the size of~$G^{K}$ cannot be bound by that of~$G$ (e.g., when~$|K| = \Theta(\sqrt{|V(G)|})$ and every vertex in~$S$ has a constant degree).
We can take an indirect approach for Helly circular-arc graphs: check whether~$G$ is a circular-arc graph when it is ambiguous (Theorem~\ref{lem:split-cag-and-hcag}) or~$G^{+}$ otherwise.



\bibliographystyle{plainurl}
\bibliography{references}

\end{document}